\documentclass[11pt, reqno]{amsart}

\usepackage[utf8]{inputenc}

\usepackage[colorlinks=true,pagebackref,hyperindex]{hyperref}
\usepackage{mathrsfs} 
\usepackage[all]{xy}
\usepackage{color}
\usepackage{amsfonts}
\usepackage{amscd}
\usepackage{amssymb,latexsym,amsmath,amscd, graphicx} 
\usepackage{bbm}

\definecolor{darkgreen}{rgb}{0.0, 0.4, 0.0}

\definecolor{cyan}{cmyk}{1,0,0,0}

\newcommand{\cb}{\color{blue}}
\newcommand{\cm}{\color{magenta}}

\newcommand{\bdg}{\begin{dg}}
\newcommand{\edg}{\end{dg}}

\newcommand{\cre}{\color{red}}

\newtheorem{tm}{Theorem}[subsection]
\newtheorem{lm}[tm]{Lemma}
\newtheorem{pr}[tm]{Proposition}
\newtheorem{rmk}[tm]{Remark}
\newtheorem{cor}[tm]{Corollary}

\newtheorem{fact}[tm]{Fact}
\newtheorem{??}[tm]{Question}
\newtheorem{defi}[tm]{Definition}

\newtheorem{setup}[tm]{Set-up}

\newcommand{\ben}{\begin{enumerate}}
\newcommand{\een}{\end{enumerate}}
\newcommand{\bit}{\begin{itemize}}
\newcommand{\eit}{\end{itemize}}
\newcommand{\beq}{\begin{equation}}
\newcommand{\eeq}{\end{equation}}
\newcommand{\la}{\label}

\newcommand\ci{\cite}

\font\tenmsb=msbm10
\font\sevenmsb=msbm7
\font\fivemsb=msbm5
\newfam\msbfam
\textfont\msbfam=\tenmsb
\scriptfont\msbfam=\sevenmsb
\scriptscriptfont\msbfam=\fivemsb
\def\Bbb#1{{\fam\msbfam #1}}
\font\teneufm=eufm10
\font\seveneufm=eufm7
\font\fiveeufm=eufm5
\newfam\eufmfam
\textfont\eufmfam=\teneufm
\scriptfont\eufmfam=\seveneufm
\scriptscriptfont\eufmfam=\fiveeufm
\def\frak#1{{\fam\eufmfam\relax#1}}

\newcommand{\lorw}{\longrightarrow}

\newcommand\rat{{\Bbb Q}}

\newcommand\oql{\overline{\Bbb Q}_\ell}
\newcommand\comp{{\Bbb C}}

\newcommand\zed{{\Bbb Z}}

\newcommand\s{\sigma}

\newcommand\e{\epsilon}

\newcommand{\w}[1]{\widetilde{#1}}

\newcommand{\m}[1]{\mathcal{#1}}
\newcommand{\ms}[1]{\mathscr{#1}}

\newcommand{\ptd}[1]{ \,^{p}\!\tau_{ \leq {#1} } }
\newcommand{\ptu}[1]{ \,^{p}\!\tau_{ \geq {#1} } }

\newcommand{\hh}{f}

\newcommand{\ph}{^p\!H}
\newcommand{\Gm}{{\mathbb G}_m}

\newcommand{\disk}{\Delta}

\newcommand{\SP}{S}

\title[Perverse Leray filtration and specialization]{Perverse Leray filtration and specialization \\
with applications to the Hitchin morphism}
\author{ Mark Andrea A. de Cataldo}
\address{Mark Andrea A. de Cataldo, Stony Brook University}
\email{mark.decataldo@stonybrook.edu}

\begin{document}

\begin{abstract} 
We initiate and develop a framework to handle the specialization morphism as a filtered morphism for the perverse, and for the perverse Leray filtration, on the cohomology with constructible coefficients of varieties and morphisms parameterized by a curve.  As an application, we use this framework to carry out a detailed study of filtered specialization for  the Hitchin morphisms associated with the compactification of Dolbeault moduli spaces in \ci{decomp2018}.
\end{abstract}

\maketitle

\tableofcontents


\section{Introduction, notation  and preliminaries}\la{intronotpre}$\;$

Let $v: Y\to S$ be a morphism into a connected nonsingular curve, let $s\in S$ be a point, let $G \in D^b_c(Y)$
be a bounded constructible complex, let $t \in \SP$ be a suitably general point (in a more algebraic set-up, the geometric generic point 
of the curve, or of an  Henselian trait),  let $Y_s$ and $Y_t$ be the corresponding
fibers. 

It is natural, and of fundamental importance, to compare 
the cohomology groups $H^j(Y_s, G_{|Y_s})$ and $H^j(Y_t, G_{|Y_t}).$
One classical example, is the study of Lefschetz pencils.
Similarly, if $X\to Y$ is an $S$-morphism, then it is equally  important to compare, for $F\in D^b_c(X)$, the  cohomology groups 
$H^j(X_s, F_{|X_s}) \to H^j(X_t, F_{|X_t})$.

 When it is defined, the specialization morphism
$H^j(Y_s, G_{|Y_s}) \to H^j(Y_t, G_{|Y_t}),$ and similarly for $X_s$ and $X_t$,
is a key tool for this comparison.

The paper \ci{dema} studies the Hitchin $S$-morphism $f: X\to Y$  for a smooth and projective family  $\m{X}/S$
and  establishes that the resulting specialization morphisms in intersection cohomology exist, and that they are filtered isomorphisms
for the respective perverse Leray filtrations. The two authors realized that there seems to be no
available discussion in the literature of the specialization morphism as a filtered morphism, so that they developed
a criterion for having such a filtered isomorphism that worked in the context of the Hitchin morphism associated with a family of smooth
varieties over a base.

In this paper, we initiate and develop a general  framework to study the specialization morphism
as a filtered morphism for the perverse (Leray) filtration. We then apply such a framework to a more detailed study of the Hitchin morphism.

\subsection{Motivation and outline of the results}\la{intro}$\;$

Let us consider the toy model situation  in Remark \ref{ert}  based on the Set-up \ref{setzxw}, where we start with a smooth morphism
 $f:X^o\to Y^o$ over a
a smooth  curve $S$, and we compactify $f$ by adding divisors $Z$ and $W$, to get $f: X^o \cup Z  = X \to Y= Y^o\cup W$, 
so that  all resulting morphisms are smooth and proper, except for $X^o$ and $Y^o$ over $S$. Then  the morphisms of long exact sequences
(\ref{gys})
for the relative singular cohomology of the  pairs  $(X_{t,s},X^o_{t,s})$ and $(Y_{t,s}, Y^o_{t,s})$ relating  any two point $s,t \in S$,  is   a filtered isomorphisms for the Leray spectral sequences
for the morphisms $f_t: X_t\to Y_t$ and $f_s:X_s \to Y_s$.  

What happens when the morphisms are not proper, the varieties are singular,
we take coefficients in an arbitrary bounded constructible complex of sheaves in $D^b_c(X)$, and we consider the perverse Leray filtration?
This is the question addressed in this paper.

First of all, why the perverse Leray filtration?
The middle perversity $t$-structure is more suitable to study singularities, so we focus on the perverse (Leray)  filtration, instead
of the Grothendieck (Leray) filtration.  In principle the Leray filtration can be studied with the methods of this paper, but 
the results  I can think of are much weaker, essentially due to the fact that the vanishing/nearby cycle functors have several important perverse $t$-exactness properties, without having a counterpart for the usual standard $t$-structure.

Let us focus on an arbitrary situation $X/Y/S$ with $K\in D^b_c(X)$, with $s \in S$ a point and $t \in S$ a suitably general point (the geometric generic point of the curve/Henselian trait,  if the reader prefers that language). 
In this  general context,  the specialization morphism (``from a special point to a general  point") 
$H^*(X_s, K_{|X_s}) \to H^*(X_t, K_{|X_t})$ is not even defined.
This is due to the failure of the relevant base-change morphisms to be isomorphisms. Even when the specialization morphism is defined, its perverse filtered counterpart may fail to be well-defined.

The purpose of this paper is to develop a framework where these questions can be studied systematically.
What follows is a list of some of the outcomes of this study, presented  in a weaker and less complete  form with respect to what is found in the body of the paper,  so that the reader may get an idea of the techniques introduced  and of the  results that are proved.

Let us start with some of the preparatory results in \S\ref{prep}.
In the set up of  morphisms $X\to Y \to S$ and of the specialization to a point $s$ on the curve $S$,
the  fibers $X_s$ and $Y_s$ over the  special point  are Cartier divisor. A key point is to study, in the more general set-up
of a Cartier divisor $T'$ inside a variety $T$,   the failure of commutativity of the perverse truncation functors $\ptd{k}$ with the restriction $i^*$ to  $T'$,
when applied to a complex $G\in D(T)$. This is codified in the failure  to be  isomorphisms of a certain  morphisms  that we introduce and denote by $\delta$ (\ref{j11a}).
Proposition \ref{iotanoc} says that the  $\delta$ are  isomorphisms when $G$ has no constituents (e.g. no non-trivial  direct summands) supported on
the divisor $T'$.
Lemma \ref{kij}, which is placed   in the context of  specialization,  gives a criterion for the $\delta$ to be isomorphisms in terms of the vanishing 
$\phi G=0$
of the  vanishing cycles of the complex $G$.
Proposition \ref{crux} is another criterion for the $\delta$ being isomorphisms for a complex $f_* F$, when we are in a  projective morphism $f: X\to Y$ situation (not necessarily over a curve $S$), where
 a  Cartier divisor on $X$, pulled back from $Y$, has suitable transversality conditions with respect to the complex $F$ on $X$: essentially, one requires 
the complex $F$  on $X$ to be semisimple and to stay  semisimple after restriction  to the Cartier divisor.   Corollary \ref{vbre}
shows that this kind of transversality can be achieved in a simple normal crossing singularities situation, where the  restriction does  not stay semisimple.

As it may be clear by now, the main goal of this paper is to identify conditions that ensure that  specialization morphisms are defined, and, when they are defined, that they are  isomorphisms, and,  in the filtered context, that they are filtered isomorphisms.  

Following the preparation in \S\ref{prep}, concerning the morphisms $\delta$, we zero in on  the problem by defining the perverse  filtered version of the specialization morphism (when it exists),  and by 
offering some criteria  in the main section  \S\ref{spsfi} of this paper.

Let me know state a version of the main Theorem \ref{crit} which is a list of different criteria for  having well defined specialization  filtered isomorphisms. Let us stress that when $v$ is not proper, without additional constraints on the situation,
the specialization morphism may fail to be well-defined.
We are in the situation $f:X \to Y,$ $v:Y\to S$, $F \in D^b_c(X)$. 
In order to try to convey the flavor of
the theorem, we offer  two different sets of conditions, each of which is  a sufficient set of conditions: 
$f$ and $v$ are proper, and $\phi(F)=0$;  $f$ proper, $F$ semisimple and $\phi(F)=0$.

We apply these results  to a compactification of the morphism $f$ as above, i.e. as in Set-up  \ref{setzxw}. 
In this case, we prove  Theorem \ref{sonoloro}, i.e that, under suitable sets of hypotheses, the  situation of the toy model discussed at the beginning of this section can be reproduced in its entirety: the specialization morphisms for $X_s,X_t$, $X^o_s,X^o_t$ and $Z_s, Z_t$   are filtered isomorphisms 
compatibly with the restriction and Gysin morphisms (which one needs to show are well-defined) stemming from the inclusions.

 Finally, Theorem \ref{lezse} is our application of the methods of this paper, and especially of Theorem \ref{sonoloro}, to the compactification of Dolbeault moduli spaces
 constructed in \ci{decomp2018}, thus showing that the various criteria developed in this paper can actually be implemented in a highly non-trivial and geometrically interesting situation. Perhaps surprisingly, this is true whether we consider intersection cohomology,
 or singular cohomology.

\subsection{More precise outline of the contents of the paper}\la{more precise}$\;$

\S\ref{intronotpre} includes this introduction \S\ref{intro}, lists the general notation \S\ref{gennot} and then discusses the formalism
of the vanishing/nearby cycles  functors;
more precisely: 
\S\ref{vncf} summarizes the parts of the formalism of vanishing/nearby cycles  functors that we need in this paper;
in particular, Fact \ref{psphit}, and (\ref{r00}) are of key importance;
\S\ref{neob} explains in some detail the geometric intuition behind the nearby cycle functor.

\S\ref{prep} is preparatory in nature: it introduces and discusses the morphisms of type $\delta$;
these are key to this paper since their being isomorphisms is necessary to the
specialization morphisms being filtered isomorphisms for the perverse (Leray) filtration. The heart of this paper, i.e. the discussion of the specialization morphism
as a filtered morphism for the perverse Leray filtration, is \S\ref{spsfi}. In order to carry out that discussion
we need to measure the failure of the restriction-to-the-special-fiber functor to commute with perverse truncation.
This failure is measured by the morphisms of type $\delta$, which are defined in Lemma \ref{j11l1} and Remark
\ref{not1z}, in the more general context of effective Cartier divisors (the special fiber being one such).
Proposition \ref{iotanoc} gives a sufficient condition for the morphisms of type $\delta$ to be isomorphisms.
Lemma \ref{kij} establishes the key facts we need when dealing with the morphisms of type $\delta$ together
with the vanishing/nearby cycle functors; in particular, the vanishing of the vanishing cycle functor gives a sufficient condition for the morphisms of type $\delta$ being isomorphisms.
Proposition \ref{crux} provides another such criterion under the assumption that certain  relative hard Lefschetz symmetries are in place; Corollary \ref{vbre} ensures that if we are in a simple normal crossing divisor situation,
then such symmetries are in place.

The aforementioned framework is developed   in \S\ref{spsfi}. We refer to the beginning of that section
for a more detailed account of its contents. Here we simply list   the main points.
The set-up is the one of an $\SP$-morphism $X\to Y$, where $\SP$ is a nonsingular connected curve,
and $s \in \SP$ is a point, the ``special" point.
The definition of the specialization morphism as a filtered morphism for the perverse Leray filtration 
is  contained in Definition \ref{defspf1}. The main Theorem of this paper is
Theorem \ref{crit}, which establishes various criteria for this morphisms to exist and to be a filtered isomorphism.
On of the main themes here is to work with non proper structural morphisms $X\to \SP$ and $Y\to \SP$, 
for in this case, the base change morphisms are not isomorphisms in general. Compactifying the situation is
one traditional way to circumvent this issue; this introduces additional base change issues, and Proposition
\ref{r54}  provides criteria to resolve them. Once a compactification is in place, one has the long exact sequence of the resulting triple (boundary, compactification, original space), and Theorem \ref{sonoloro}
provides three criteria ensuring that the resulting three specialization morphisms gives rise to an isomorphism
of filtered long exact sequences associated with the special and general triples.

\S\ref{apell} applies the abstract framework developed in \S\ref{spsfi} in the case of the Hitchin morphisms
arising from the compactification of Dolbeault moduli spaces introduced in \ci{decomp2018}.
We refer to the beginning of that section for a detailed outline of the contents of this section.
The main result is Theorem \ref{lezse}, to the effect that we get the desired isomorphism
of filtered long exact sequences associated with the special and general triples stemming from the particular compactification \ci{decomp2018}. Certain preliminary results concerning descending certain 
properties along $\Gm$-quotients, which could be of general independent interest, are established along the way.

\begin{rmk}\la{otherfields}
{\rm ({\bf Other algebraically closed fields, generic geometric points vs general points})}  We have chosen to work with  vanishing/nearby cycles wrt a morphism into  a nonsingular curve, over the field of complex numbers, with the classical topology, and with 
finite algebraic Whitney stratifications (which play a role only in the background, by merely existing and having the usual properties).
While this is mostly a matter of expository style, we also work in a more global (not over a curve/Henselian trait) context in parts of \S\ref{prep}.
By complementing the references given in \S\ref{vncf} with \ci[\S4]{il2}, which also deals with vanishing/nearby cycles and the perverse t-structure, the exposition and the results
 of this paper remain valid, mutatis mutandis,  for varieties over algebraically closed fields of characteristic zero and for the  \'etale cohomology
with coefficients in $\oql$, for any prime $\ell$;
for example, the role played in this paper
by a suitably general point $t$ on a nonsingular curve $S$  is now played by the geometric generic point of $S$, or by the geometric generic point of an Henselian trait. Similarly, except for \S\ref{apell}, where in the application to the compactification of the Hitchin morphism we need to consider quotients by finite groups on which we have little control, the results remain 
valid over an algebraically closed field of positive characteristic   and  $\oql$-coefficients, with $\ell$ not dividing the characteristic of the field.
 \end{rmk}

{\bf Acknowledgments.}
The author  thanks: Michel Brion, Victor Ginzburg, Jochen Heinloth, Luca Migliorini,  J\"org Sch\"urmann and  Geordie Williamson for useful suggestions. Special thanks to Davesh Maulik, with whom the author has discussed various preliminary versions of Theorem \ref{lezse}. The  author, who is partially supported by N.S.F. D.M.S. Grants n. 1600515 and 1901975, would like to thank the Freiburg Research Institute for Advanced Studies  for the perfect working conditions; the research leading to these results has received funding from the People Programme (Marie Curie Actions) of the European Union's Seventh Framework Programme
(FP7/2007-2013) under REA grant agreement n. [609305].  Finally,  we thank very warmly  the referee for pointing out some inaccuracies in the first submission.


\subsection{Notation}\la{gennot}$\;$

By variety, we mean a separated scheme of finite type over the field of complex numbers
$\comp$.  By point, we mean a closed point. See \ci{bams} for a quick introduction  and for standard references concerning  the constructible derived category and other
concepts in this subsection.

Given a variety $Y,$ we denote by $D^b_c(Y)$ the constructible bounded derived category of  sheaves of rational vector spaces on $Y,$ and by $DF(Y)$ its filtered variant \ci{il,bbd}. We endow $D^b_c(Y)$
with the middle perversity $\frak{t}$-structure. 
A functor that is exact with respect to this $\frak t$-structure is said to be 
$\frak t$-exact. 
The full subcategory of perverse sheaves is denoted by $P(Y)$.
We employ  the following standard notation for the objects associated with this $\frak{t}$-structure:
the full subcategories 
$^p\!D^{\leq j}(Y)$ and ${^p\!D}^{\geq j}(Y)$, $\forall j \in \zed$, and
${^p\!D}^{[j,k]}(Y): =    {^p\!D}^{\geq j}(Y) \cap    {^p\!D}^{\leq k}(Y),$ $\forall j\leq k \in \zed$,
of $D^b_c(Y)$; the truncation functors 
$\ptd{j}: D^b_c(Y) \to  {^p\!D}^{\leq j}(Y)$ and  $\ptu{j}: D^b_c(Y) \to {^p\!D}^{\geq j}(Y)$;  the perverse cohomology functors
$\ph^j: D^b_c(Y) \to P(Y)$. 
We denote derived functors using the un-derived notation, e.g. if $f:X\to Y$
is a morphism of varieties, then the derived direct image (push-forward) functor  $Rf_*: D^b_c(X) \to D^b_c(Y)$ is denoted by $f_*$, etc.
Distinguished triangles  in $D^b_c(Y)$ are denoted $G' \to G \to G'' \rightsquigarrow$. 
At times, we drop the space variable $Y$ from the notation.

The following operations preserve constructibility of complexes:
ordinary and extraordinary push-forward and pull-backs, hom and tensor product,
Verdier duality, nearby and vanishing cycles.


When we write a  Cartesian diagram of morphisms of varieties: 
\beq\la{amb}
\xymatrix{
X'  \ar[r]^g \ar[d]^f &     X \ar[d]^f
\\
Y'  \ar[r]^g &   Y,
}
\eeq
the ambiguities arising by having denoted different arrows by the same  symbol,
are automatically resolved by the context in which they are used; e.g. when we write the base change morphism
of functors, the expression  $g^*f_* \to f_*g^*$ is unambiguous.

The $k$-th  (hyper)cohomology
groups of $Y$ with coefficients in $G \in D^b_c(Y)$ are denoted by $H^k(Y,G)$. The complex computing
this cohomology is denoted by $R\Gamma (Y,G)$ and it lives in the bounded derived category
$D^b_c(pt)$
whose objects are    complexes of vector spaces with cohomology given by finite dimensional rational vector spaces.

The filtrations we consider are finite and increasing.  A sequence of morphism 
$G_\bullet:= (\ldots \to G_n \to G_{n+1} \to \dots)$
in $D^b_c(Y)$
subject to $G_i =0$ $\forall i \ll 0$, and  to $G_i \stackrel{\sim}\to G_{i+1}$ $\forall i \gg 0$,
gives rise to  objects $(G,\ms{F}) \in DF(Y)$ and  $(R\Gamma (Y, G), \ms{F}) \in D^b_c(pt).$
We call such a sequence of morphisms a system in $D^b_c(Y).$ 
There is the evident notion of morphism of systems
$G_\bullet \to G'_\bullet$ and, for each such, 
the resulting morphisms  $(G,\ms{F}) \to (G', \ms{F})$ in $DF(Y),$ and $(R\Gamma (Y, G), \ms{F}) \to (R\Gamma (Y, G'), \ms{F})$ in $DF(pt)$.
Given $G\in D^b_c(Y)$ we have the  system of perverse truncation morphisms:
$
\ldots \to \ptd{n} G \to  \ptd{n+1} G \to \ldots \to G.
$
A morphism $G\to G'$ in $D^b_c(Y)$ gives rise to a morphism of systems
$
\ptd{\bullet} G \to \ptd{\bullet} G',
$
which  gives rise to a morphisms of filtered objects:
\beq\la{lmn}
\xymatrix{
\left( R\Gamma (Y, G), P) \right) \ar[r] &  \left( R\Gamma (Y, G'), P)\right) & {\rm in}\; DF({pt}).
}
\eeq
The filtration $P$ is called the perverse filtration.
There is also the evident notion of systems of functors; e.g. the truncation functors. 

Let $f:X\to Y$ be a morphism of varieties,  let  $F\in D^b_c(X)$ and  let  $G=f_* F$.  Then we have  the filtered objects $(R\Gamma (X,F), P)\neq (R\Gamma(X,F), P^f) := (R\Gamma (Y, f_* F),P).$ The filtration $P^f$ is called the perverse Leray filtration associated with $f.$
Given a morphism $F \to F'$ in $D^b_c(X),$ the analogue of (\ref{lmn}) holds. We have the corresponding objects $P^f_j H^k(X,F)$, ${\rm Gr}^{P^f}_j H^k(X,F)$.

If a statement is valid for every value of  an index, e.g. the degree of a cohomology group, or the step of a filtration, then we denote such an index by a bullet-point symbol, or by a star symbol, e.g.  $H^\bullet (X,F)$, $P_{\bullet}$, $\ptd{\bullet}$,  $P_\star H^\bullet (X,F)$, ${\rm Gr}^{P_f}_\star H^\bullet (X,F)$.

We employ the following convention for shifts of filtered increasing filtrations: 
\beq\la{conv}
\ms F(j)_\bullet:= \ms{F}_{\bullet -j}.
\eeq

We have  the evident and  equivalent relations between truncations and shifts:
\beq\la{tsh}
[\star] \circ \ptd{\bullet}= \ptd{\bullet -\star} \circ [\star], 
\qquad
\ptd{\bullet} \circ  [\star]  = [\star] \circ \ptd{\bullet +\star},
\eeq
which are valid for  $\ptu{\bullet}$ and $\ph^{\bullet}$ as well.

The category  $P(Y)$  of perverse sheaves on $Y$ is Abelian and  Artinian, so that the Jordan-Holder theorem holds in it.
The constituents of a non-zero perverse sheaf $G \in P(Y)$ are the isomorphisms classes of the
perverse sheaves appearing  in the unique and finite
collection of non-zero simple perverse sheaves appearing as  the quotients  in a Jordan-Holder filtration of $G.$
The  constituents of  a non-zero complex $G\in D^b_c(Y)$ are defined to be the constituents of all of its non-zero perverse
cohomology sheaves.

In general, we drop decorations (indices, parentheses, space variables, etc.) if it seems harmless in the context.

Given a morphism of varieties $X\to Y$ and a point $y \in Y$,  we denote by $X_y$ the fiber over $y$ in $X$.

We are going to  use the nearby/vanishing cycle functors. See \S\ref{vncf} for the basic facts.


For what follows, namely the Decomposition Theorem,  see:   \ci{bbd}, \ci{saito}, \ci{htam}, \ci{mo,sa}. For a discussion, further references, and a  generalization of \ci{mo} to the proper case with rational coefficients see  \ci{dess}. 

\begin{tm}\la{dt}
{\rm ({\bf Decomposition  and Relative Hard Lefschetz theorems})}

Let $f:X\to Y$ be a proper morphism of varieties and let $F \in P(X)$ be semisimple. Then $f_* F$ is semisimple.

If $f$ is projective, then the Relative Hard Lefschetz holds: the choice of an $f$-ample line bundle induces isomorphisms
${\ph^{-\bullet}} (f_* F) \stackrel{\simeq}\to
{\ph^{\bullet}}(f_*(F)$.
\end{tm}

We  define the intersection complex $IC_T$ of a variety $T$ as the direct sum of the intersection complexes
of its irreducible components $T_j$: $IC_T:= \oplus_j IC_{T_j}$.  Then, as it is shown in \ci{dejag}:  $IC_T \in P(T)$ is a semisimple perverse sheaf; it is the intermediate extension from the smooth locus of $T$ of the direct sum of the constant sheaves on each component
shifted by its dimension (see \ci{wu} for a topological characterization); the Decomposition Theorem, Relative Hard Lefschetz Theorem, and allied Hodge-Theoretic facts
hold for it and its cohomology.  The complex $IC_T$  underlies a polarizable Hodge module. 
The simple objects in $P(Y)$ are the intersection complexes  $IC_T(L)$, where $T\subseteq Y$ is closed and irreducible
and $L$ is an irreducible local systems on some Zariski dense open subset $T^o \subseteq T^{\rm reg}$.

We define the topologist's intersection complex
as follows: 
$\m{IC}_T:= \oplus_j IC_{T_j}[\dim T_j]$. 

We define the intersection cohomology groups  by setting $I\!H^\bullet (T):= H^\bullet (T, IC_T)$;  they 
start in degree $-\dim{T}$. We define the topologist's intersection cohomology groups  by setting $\m{I\!H}^\bullet (T):= H^\bullet (T, \m{IC}_T)$;  they 
start in degree zero.
E.g. for $T= o \coprod \Lambda$ the disjoint union of a point and a line, we have:
$\rat_T = \rat_o \oplus \rat_\Lambda$;  
$IC_T=\rat_o \oplus \rat_\Lambda [1]$;  $\m{IC}_T=\rat_o \oplus \rat_\Lambda$; $\m{IH}^\bullet (T) =H^\bullet(\rat_T) =H^\bullet (o) \oplus H^\bullet (\Lambda)$; $I\!H^\bullet (T) = H^\bullet (o) \oplus H^{\bullet +1} (T)$.

In order to avoid  repeating naturality-type statements, we employ systematically the following notation.
The symbol ``$=$" denotes either equality, or a canonical isomorphism. 
The symbol
$\stackrel{\simeq}\to$ and  denotes the fact that a canonical arrow is, in the context where it appears, an isomorphism.
The symbol $\cong$ 
denotes an isomorphism; if the direction of the arrow is relevant, we write $\stackrel{\simeq}\to$. E.g.: the base change canonical isomorphism reads: $g^!f_*= f_* g^!$;
the proper base change isomorphism for $f$ proper reads:  $g^* f_* \stackrel{\simeq}\to f_* g^*$.


\subsection{The vanishing and nearby cycles formalism }\la{vncf}$\;$

This section is an expanded version of \ci[\S2.2]{dema}.
Standard references  for this section are \ci[XIII, XIV]{sga72}  and \ci[Ch. 8,10]{kash}.
See also \ci{il2} (cf. Remark \ref{otherfields}).


Let  $v=v_Y: Y\to S$
be a morphism of varieties, where $S$ is a nonsingular curve. Let $s \in S$ be a point, usually called the special point.

We have the exact  functors  of triangulated categories:
\beq\la{fctz}
i^*,i^!, \psi=\psi_v, \phi=\phi_v:  D^b_c(Y) \to D(Y_s),
\eeq
where, $i: Y_s \to Y$ is the closed embedding of the (special)  fiber of $v$ over $s$,  $\phi$ is the vanishing cycle functor and $\psi$ is the nearby cycle functor. 
The functors $\phi$ and $\psi$ depend on $v$; e.g.   \ci[Example 2.0.5]{elm}; this is not an issue in this paper.

We have  the two, Verdier dual, canonical distinguished triangle of functors: (we often write $\s$ instead of $\s^*$ and $\s^!$)
\beq\la{1p}
\xymatrix{
i^*[-1]  \ar[r]^-{\s^*} & \psi [-1]  \ar[r]^-{\rm can} & \phi   \ar@{~>}[r]   &, &
\phi \ar[r]^-{\rm var} & \psi[-1]  \ar[r]^-{\s^!} & i^![1]   \ar@{~>}[r] &.
}
\eeq

\begin{fact}\la{psp0}

{\rm ({\bf 
$\frak{t}$-exactness for $\psi [-1]$ and $\phi$})} 
The functors $\psi [-1]$  and $\phi$ are $\frak t$-exact and commute with Verdier duality. We thus  have the following canonical identifications: \beq\la{can1}
\xymatrix{
\ptd{\bullet} \phi =  \phi \ptd{\bullet}; & \ptd{\bullet} \psi [-1] =  \psi [-1] \ptd{\bullet};&
 \mbox{ditto for $ \ptu{\bullet}$ and  $\ph^{\bullet}$}.
}
\eeq

\end{fact}

 
 The following is a key property of the vanishing cycle functor.
 
 \begin{fact}\la{psp}
{\rm ({\bf 
Smooth morphisms and vanishing of $\phi$})}
If $v: Y\to S$ is smooth over a neighborhood of $s$, and  $G\in D^b_c(Y)$  has locally constant cohomology sheaves  near $s,$ 
then $\phi\,G=0 \in D(Y_s)$. See \ci[XIII, 2.1.5]{sga72}. In the special case where  $v=id_S$, this implies that 
 a complex $G \in D(S)$ has constant cohomology sheaves  near $s$ if and only if  $\phi G=0$.
\end{fact}

We recall, for later use,  the following simple

\begin{lm}\la{silm} {\rm ({\bf 
Vanishing of $\phi$ implies no constituents})}
Let $G \in D^b_c(Y).$  If $\phi G=0,$ then no constituent of $G$ is supported on $Y_s.$
\end{lm}
\begin{proof}
See \ci[Lemma 3.1.5]{dema} 
\end{proof}

\begin{rmk}\la{gt5}
The converse to Lemma \ref{silm} is false; see \ci[Ex. VIII.14]{kash} (Lefschetz degeneration to an ordinary double point).\end{rmk}

 \begin{fact}\la{gh1}
The composition $i^*[-1] \to \psi[-1] \to i^![1]$ yields a natural morphism of functors $D^b_c(Y) \to D(Y_s)$:
\beq\la{gh}
\xymatrix{
{\nu}:
i^*[-1] \ar[r] &   i^![1].
}
\eeq
The morphism $\nu$ 
coincides with the morphism obtained via Verdier's specialization functor, (cf. \ci{jorg}, for example), so that it depends  only on the closed embedding $Y_s\to Y,$
i.e. it is independent of $v.$ 
If $\phi G=0$, then $\s^*(G), \s^!(G)$ and $\nu(G)$  are  isomorphisms. 
\end{fact}

\begin{fact}\la{psphit}
({\bf 
Base change diagrams for $\psi$ and $\phi$})
Let $f:X\to Y$ be an $S$-morphism and let $v_X= v_Y \circ f:X\to S$ be the structural morphism.
The various base change natural transformation associated with  $i$ and $f$ give rise to 
functorial morphisms of distinguished triangles as follows. See \ci[XIII, 2.1.7]{sga72}.

\ben

\item\la{pp3a}
The one in $D(X_s)$  stemming from the  morphism of functors $i^*f^!\to f^!i^*$, and from the 
 isomorphism of functors $f^*i^* \stackrel{=}\to i^*f^*,$ respectively: (we write $\phi_Y$ instead of $\phi_{v_Y},$ etc.)
\beq\la{fdt2bis}
\xymatrix{
 i^*[-1]f^!    \ar[r]^-{\s \circ f^!} \ar[d]  &   \psi_X [-1]  f^!    \ar[r]  \ar[d]  &  \phi_X  f^!  \ar[d]    \ar@{~>}[r] & &
  i^*[-1] f^*    \ar[r]^-{\s \circ f^*}  &   \psi_X [-1] f^*    \ar[r]  &  \phi_X  f^*      \ar@{~>}[r] &
\\
 f^! i^* [-1] \ar[r]^-{\s \circ f^!}    &  f^! \psi_Y [-1]     \ar[r]    & f^! \phi_Y   \ar@{~>}[r]  &, &
  f^* i^* [-1]  \ar[r]^-{\s \circ f^!} \ar[u]^-=     &  f^* \psi_Y  [-1]    \ar[r]    \ar[u]  & f^* \phi_Y \ar[u]   \ar@{~>}[r]  &.
}
\eeq
When $f$ is \'etale, so that $f^*=f^!,$
the two resulting morphisms of triangles are isomorphisms, inverse to each other.

\item\la{pp3b}
The one in $D(Y_s)$ stemming from   the base change morphism of functors $i^*f_* \to f_* i^*,$ and from 
the base change isomorphism
$f_!i^* \stackrel{=}\to i^*f_!:$
\beq\la{fdtbis}
\xymatrix{
i^* [-1] f_*   \ar[r]^{\s\, \circ f_*} \ar[d] &   \psi_Y [-1]  f_*   \ar[r] \ar[d]  &   \phi_Y f_* \ar[d]  \ar@{~>}[r] & &
i^* [-1] f_!   \ar[r]^{\s\, \circ f_!}  &   \psi_Y [-1]  f_!   \ar[r]   &   \phi_Y f_!   \ar@{~>}[r] &
\\
f_* i^* [-1]  \ar[r]^{f_* \circ \, \s}   &  f_* \psi_X  [-1]  \ar[r]   & f_* \phi_X   \ar@{~>}[r]  &, &
f_! i^* [-1] \ar[r]^{f_! \circ \, \s}  \ar[u]^-=  &  f_! \psi_X [-1]   \ar[r] \ar[u]   & f_! \phi_X  \ar[u]   \ar@{~>}[r]  &.
}
\eeq
When $f$ is proper, so that $f_!\stackrel{\simeq}\to f_*$,
the two resulting morphisms of triangles are isomorphisms, inverse to each other.

\item\la{p!}
We dualize the four diagrams associated in  (\ref{fdt2bis}) and (\ref{fdtbis}), and obtain the four analogous diagrams: 
\beq\la{fd0}
\xymatrix{
 \phi_X f^*    \ar[r]   &   \psi_X[-1] f^*    \ar[r]^-{\s \circ f^*}   &  i^! [1] f^*      \ar@{~>}[r] & &
  \phi_X f^!    \ar[r] \ar[d]  &   \psi_X[-1]  f^!    \ar[r]^-{\s \circ f^!}  \ar[d]  &i^! [1] f^!    \ar[d]^-=    \ar@{~>}[r] &
\\
 f^* \phi_Y \ar[r]  \ar[u]   &  f^* \psi_Y [-1]    \ar[r]^-{f^* \circ \s}   \ar[u]    & f^* i^![1]   \ar@{~>}[r]   \ar[u]   &, &
  f^! \phi_Y \ar[r]     &  f^! \psi_Y[-1]    \ar[r]^-{f^! \circ \s}      & f^! i^![1]   \ar@{~>}[r]  &,
}
\eeq
which are isomorphisms inverse to each other when $f$ is \'etale, and 
\beq\la{fd1}
\xymatrix{
 \phi_Y f_!    \ar[r]  &   \psi_Y[-1] f_!    \ar[r]^-{\s \circ f_!}   &  i^! [1] f_!     \ar@{~>}[r] & &
  \phi_Y f_*    \ar[r]  \ar[d]  &   \psi_Y [-1] f_*    \ar[r]^-{\s \circ f_*}  \ar[d]  &  i^! [1] f_*   \ar[d]^-=    \ar@{~>}[r] &
\\
 f_! \phi_X \ar[r]  \ar[u]   &  f_! \psi_X [-1]    \ar[r]^-{f_! \circ \s}    \ar[u]   & f_! i^![1]  \ar[u]  \ar@{~>}[r]  &, &
  f_* \phi_X \ar[r]      &  f_* \psi_Y[-1]     \ar[r]^-{f_* \circ \s}      & f_*i^![1]  \ar@{~>}[r]  &,
}
\eeq
which are isomorphisms  inverse to each other  when $f$ is proper.

\item
by combining  the  l.h.s. square commutative diagram in (\ref{fdtbis}) with the r.h.s. square commutative diagram in    (\ref{fd1}),  we obtain the following commutative diagram of morphisms of functors:
\beq\la{r00}
\xymatrix{
i^*[-1] f_* \ar[d]    \ar[r]   & \psi[-1] f_* \ar[r]    \ar[d]    & i^![1] f_* \ar[d]^-= 
\\
f_* i^*[-1]      \ar[r]   & f_* \psi[-1]  \ar[r]        &      f_* i^![1],
}
\eeq
where the compositions of the horizontal arrows are given by the corresponding morphims  (\ref{gh}), suitably pre/post-composed with $f_*$.
\een
\end{fact}

\subsection{Nearby cycle functor and nearby points}\la{neob}$\;$

Even if logically not necessary,  it maybe helpful to the intuition to clarify the use of the word ``nearby."

Let $t \in S$ be any point. By abuse of notation, denote by $t: t \to S$,  $t: X_t \to X$ and  $t: Y_t \to Y$ the resulting closed embeddings.
There is the  morphism  (\ref{gh}) of functors $D(-) \to D(-_t):$
\beq\la{ght}
t^*[-1]\lorw  t^![1].
\eeq

Let us start with the following

\begin{fact}\la{gi9} For what follows, see \ci[Fact 2.2.7]{dema}.
Given a   finite collection $G \in D^b_c(Y)$ 
of complexes,  there is a Zariski-dense open subset
$S^o (G) \subseteq S$ such that: the direct images ${v}_* \ptd{\bullet} G$ have locally constant cohomology sheaves, and  their formation commutes with arbitrary base change;
the $\ptd{\bullet} G$ have no constituent supported on any of the  fibers of the morphisms $v$ over $S^o  (G)$;
 the strata on  $Y,$ of a stratification
with respect to which $G$ -hence the   $\ptd{\bullet}G$- are constructible, are smooth over $S^o (G)$.
\end{fact}

\begin{defi}\la{tge}
We say that $t \in S$ is general for $G$ 
if $t \in S^o  (G)$ as in Fact \ref{gi9}.
\end{defi}

\begin{fact}\la{gogo}
Let $G\in D^b_c(Y)$.
Let $t \in S^o (G)$ be a  general point for $G.$  Then:

\ben
\item
 The natural morphism $t^* [-1] G\to t^![1]G$ is an isomorphism
in $D(Y_t).$  See \ci[Fact 2.2.5]{dema}.
\item
We have the identifications of  \ci[Fact 2.2.6]{dema}:
\beq\la{vbgfz}
t^*[-1] \ptd{\bullet} G = \ptd{\bullet } t^*[-1] G, \quad t^![1] \ptd{\bullet} G = \ptd{\bullet } t^![1] G,
\quad \mbox{ditto for  $\ptu{\bullet}$ and  $\ph^{\bullet}$.}
\eeq
\een
\end{fact}

 

\begin{rmk}\la{idspt}
In the special case where $v_Y:Y\to S$ is the identity on $S,$ we have that $i^*,i^!, \psi ,\phi: D(S) \to D(s)$ and that $t^*, t^!: D(S)\to D(t)$.
In general, we  have canonical identifications $D(s)=D({pt})= D(t),$ where  ${pt}$ is just a point, so that all three categories are 
naturally equivalent to the bounded derived category of finite dimensional  rational vector spaces. Similarly, in the filtered case:
$DF(s)=DF({pt})= DF(t)$. We use the catch-all notation $D^b_c(pt)$ and $DF(pt)$.
\end{rmk}

Given $G\in D^b_c(Y)$ and $s\in S$ our special point, we want to think: of the points $t \in S^o(G)$ as the points nearby $s;$
 of the nearby vanishing cycle as capturing the complex $G$ at points $t$ nearby $s.$ The following important fact makes this precise.

\begin{fact}\la{jo} {\rm{\bf (Fundamental isomorphism)}}
For what follows, see the fundamental identity \ci[XIV, 1.3.3.1]{sga72}. 
Let $v:Y\to S$ be the identity, let $G \in D^b_c(Y),$ 
let $t \in S^o (G)$ be a  general point for $G.$  
Choose a disk $\disk \subseteq S$ centered at $s$ and passing through $t,$  such that $\disk^*:= \disk \setminus \{s\} \subseteq S^o(G).$ Choose an universal covering of $\disk^*$ and a point $\w{t}$ on it over $t.$ This choice of data gives rise
to  isomorphisms: 
\beq\la{000}
\xymatrix{
t^*[-1] G \ar[r]^-\sim & \psi [-1]  G  \ar[r]^-\sim & t^![1] G,  &in\; D({pt}),
}
\eeq
where the composition is (\ref{ght}), hence it is independent of the choices made above, but the indicated morphisms depend   on the choice of data made above.
Note that {\rm (\ref{000})} is Verdier self-dual.
Changing the  choice of data  changes the two individual arrows by the monodromy automorphism induced
by an appropriate element in $\pi_1 (S^o(G),t).$

Let $v: Y\to S$ be any morphism and let $G \in D^b_c(Y).$ Then:

\ben
\item
By taking $v_*G \in D(S)$  in (\ref{000}), we get  collections
of isomorphisms:
\beq\la{r401}
\xymatrix{
\psi v_* G  \ar[r]^-\sim & t^* v_* G =  v_* t^* G = R\Gamma (Y_t, t^*G), & in\; D({pt}),
}
\eeq
where any two such identifications $\stackrel{\sim}\to$  differ by the action induced  by an element of
$\pi_1 (S^o(G), t)$ on the r.h.s.  The second $=$ sign is clear.  The first
can be seen by  using the  base change properties relative to the general point $t$ expounded in  \S\ref{vncf}
as follows.
The choice of $t$ general for  $G,$ made in Definition
\ref{tge}, allows us to:  
replace $\psi $ with $t^*$ (cf. (\ref{000})); use the identification $t^*v_*=v_* t^*$. 
 One cannot select  a  point $t\in S$ that is general for every $G \in D^b_c(Y).$ Because of this it is preferable to work with 
 $\psi: D(S)\to D(s)$
 instead of with a general point $t \in S.$  When working with finitely many complexes in $D^b_c(Y),$
 or with a family of complexes constructible with respect to an arbitrary fixed stratification,  we can always choose a
 point  $t \in S$ that is general for all of them.
\item
By taking $v_*\ptd{\bullet} G \in D(S)$ and  $v_*\ptd{\bullet} f_* F \in D(S)$ in (\ref{000}), we reach conclusions  analogous to part (1), namely, we obtain isomorphisms:
\beq\la{r403}
\xymatrix{
(\psi v_* G, P)  \ar[r]^-\sim & (R\Gamma (Y_t, t^*G),P), & in\; DF({pt});
}
\eeq
\beq\la{r405}
\xymatrix{
(\psi v_* f_* F, P)  \ar[r]^-\sim & (R\Gamma (X_t, t^*F),P), & in\; DF({pt}).
}
\eeq
\een
\end{fact}

\section{The morphisms of type \texorpdfstring{$\delta$}{d}}\la{prep}

This section is a necessary technical interlude on the way to \S\ref{spsfi}, where we study the behavior
of the perverse filtration relative to the specialization morphism.

The  distinguished triangles  (\ref{1p})  and  the  $\frak t$-exactness of the functors $\phi $ and $\psi [-1]$
 suggest the need to  quantify the failure of  $\frak t$-exactness for the functors $i^*[-1]$ and $i^! [1]$.

In this section we quantify this failure by introducing the morphisms of functors  of type $\delta$
(\ref{j11aa})  which play an important role starting with  \S\ref{spsfi}, diagram (\ref{ar0}). When
the morphisms of type $\delta$ are  evaluated against a complex $G\in D^b_c(Y)$,  they are isomorphisms
if and only if   perverse truncations and perverse cohomology commute with the appropriately shifted pull-backs.

The special fiber $Y_s$ in $Y$ is an effective Cartier divisor. We introduce the morphisms of type $\delta$ in the more general context of embedded effective Cartier divisors in  (\ref{j11a}). For easier bookkeeping, we use the  shifted version
(\ref{bnnb}).

Proposition  \ref{iotanoc} is a criterion, in the general context of Cartier divisors (i.e. when we are not necessarily working over a curve $S$), for the morphisms of type $\delta$ to be isomorphisms,
hence it is a criterion for  the aforementioned  commutativity of perverse truncations/cohomology
with shifted pullbacks.
 It states that if  a complex  has no constituents  supported on an effective  Cartier divisor,
then the resulting morphisms of type $\delta$ are isomorphisms.
 
Lemma \ref{kij} yields, in the context of the vanishing/nearby cycle functors (here we are working over a curve $S$, and  the Cartier divisor in question is the fiber over the  special
point $s$ )
 the important diagram (\ref{fr1z}) of morphisms of  distinguished triangles of functors   involving the morphisms of type $\delta$ needed in this paper. In particular, Lemma \ref{kij}.(3) combines
Lemma \ref{silm} (if $\phi G=0$, then $G$ has no constituents on $Y_s$) with Proposition \ref{iotanoc} (``if no constituents,
then the $\delta$'s are isomorphisms).

Proposition \ref{crux} is another general (i.e not necessarily related to a situation over a curve)  criterion of a different flavor for the $\delta$'s to be isomorphisms in the context of a complex $G=f_*F$
being a direct image under a projective  morphism $f$. It replaces the assumptions on constituents and/or on the vanishing of the vanishing cycle functor,  with the assumptions that both $F$ and $i^*F[-1]$ are semisimple; the proof uses the Relative Hard Lefschetz Theorem.

\subsection{Pull-back to Cartier divisors and truncations: the morphisms \texorpdfstring{$\delta$}{d}}\la{cdte} $\;$

In the context of the vanishing cycle functor, i.e. $v:Y \to S$ and $s \in S$,  the special fiber $Y_s$ is a Cartier divisor in the total space $Y.$ The    vanishing cycle functor $\phi$  and the  shifted nearby cycle functor
$\psi [-1]$
  are $\frak t$-exact. It is important to study the defect of $\frak t$-exactness
of the shifted restriction to the special fiber functors $i^*[-1]$ and $i^![1].$

In this section,   which amplifies \ci[\S3.1]{dema} to meet the needs of this paper, we introduce, in the general context of the closed embedding of an effective Cartier divisor, morphisms of functors that measure the aforementioned  defect.
We define the morphisms of type $\delta$ in the context of embeddings of Cartier divisors in Lemma \ref{j11l1}, and we prove a criterion for when these morphisms $\delta$ are isomorphisms in Proposition \ref{iotanoc}.

Let us emphasize that in this section we work with varieties and not with $S$-varieties.

We start by listing the needed general   $\frak{t}$-exactness properties related to embeddings of effective Cartier divisors on varieties.

\begin{lm}\la{iotatex} {\rm ({\bf Inequalities for embeddings of Cartier divisors})}
Let $\iota: T' \to T$ be a closed embedding  of varieties such that the open embedding  $T\setminus T'$ is an affine morphism
(e.g. $T'$ is an effective Weil divisor supporting an effective Cartier divisor). Then: (we omit
 the space variables)
 
 \ben
 \item The functor $\iota^*$ is  is right $\frak{t}$-exact and the functor $\iota^!$ is left $\frak{t}$-exact:
 \beq\la{isrt}
 \xymatrix{
 \iota^*: {^p\!D}^{\leq \bullet}  \to {^p\!D}^{\leq \bullet} , &  \iota^!: {^p\!D}^{\geq \bullet}  \to {^p\!D}^{\geq \bullet}.
 }
 \eeq
 \item
The functor $\iota^*[-1] $ is left  $\frak{t}$-exact and the functor $\iota^! [1]$ is 
right $\frak{t}$-exact:
\beq\la{isrt1}
\xymatrix{
\iota^*: {^p\!D}^{\geq \bullet}  \to {^p\!D}^{\geq \bullet -1}, & \iota^!: {^p\!D}^{\leq \bullet}  \to {^p\!D}^{\leq \bullet +1}.
}
\eeq
\een
In particular, we have that: $\iota^* P(T)  \subseteq {^p\!D}^{[ -1,0]}(T')$ and 
$\iota^! P(T) \subseteq {^p\!D}^{[0 ,1]} (T')$.
\end{lm}
\begin{proof}
See \ci[Lemma 3.1.1]{dema}. 
\end{proof}

The following lemma extends \ci[Lemma 3.1.2]{dema}  in the direction needed in this paper.

We denote by $\gamma_{\leq \bullet}:\ptd{\bullet -1} \to  \ptd{\bullet}$  and 
$\gamma_{\geq \bullet}: \ptd{\bullet } \to  \ptd{\bullet +1}$
 the natural morphisms. They are Verdier dual to each other.  We set:
 \[\gamma_{\leq \bullet}^* :=\gamma_{\leq \bullet} \iota^*,   \quad
 \gamma_{\leq \bullet}^! :=\gamma_{\leq \bullet} \iota^!,  \quad
 \gamma_{\geq \bullet}^* :=\gamma_{\geq \bullet} \iota^*, \quad
 \gamma_{\geq \bullet}^! :=\gamma_{\geq \bullet} \iota^!. 
 \]
 The entries of each of the two pairs 
 $(\gamma_{\leq \bullet}^*, \gamma_{\geq \bullet}^!)$ and $(\gamma_{\geq \bullet}^*, \gamma_{\leq \bullet}^!)$    are
 exchanged by Verdier  duality.
\begin{lm}\la{j11l1}
{\rm ({\bf The morphisms $\delta$})}
There are natural morphisms of distinguished triangles of functors:
\beq\la{j11a}
\xymatrix{
\ptd{\bullet -1} \iota^*  \ar@/_3pc/[dd]|-{\gamma^*_{\leq \bullet}}  \ar[r] \ar[d]^-{\delta^*_{\leq \bullet}} & \iota^* \ar[d]^-= \ar[r] & \ptu{\bullet}  \iota^*  
\ar[d]^{\delta^*_{\geq \bullet}}  \ar@/^3pc/[dd]|-{\gamma^*_{\geq \bullet}}  \ar@{~>}[r]  & &
\ptd{\bullet -1} \iota^!  \ar@/_3pc/[dd]|-{\gamma^!_{\leq \bullet}}  \ar[r] \ar[d]^-{\e^!_{\leq \bullet}} & \iota^! \ar[d]^-= \ar[r] & \ptu{\bullet}  \iota^!  
\ar[d]^{\e^!_{\geq \bullet}}  \ar@/^3pc/[dd]|-{\gamma^!_{\geq \bullet}}  \ar@{~>}[r]  &
\\
\iota^* \ptd{\bullet}   \ar[r] \ar[d]^-{\e^*_{\leq \bullet}} & \iota^* \ar[u] \ar[d]^-= \ar[r] &\iota^*  \ptu{\bullet +1}    \ar[d]^-{\e^*_{\geq \bullet}}  \ar@{~>}[r]  & &
\iota^! \ptd{\bullet -1}   \ar[r] \ar[d]^-{\delta^!_{\leq \bullet}} & \iota^! \ar[u] \ar[d]^-= \ar[r] &\iota^!  \ptu{\bullet}    \ar[d]^-{\delta^!_{\geq \bullet}}  \ar@{~>}[r] &
\\
\ptd{\bullet} \iota^*  \ar[r]  & \iota^*  \ar[r] \ar[u] & \ptu{\bullet +1}  \iota^*  \ar@{~>}[r] &, &
\ptd{\bullet} \iota^!  \ar[r]  & \iota^!  \ar[r] \ar[u] & \ptu{\bullet +1}  \iota^!  \ar@{~>}[r] &,
}
\eeq
which are exchanged by Verdier duality.

By iteration, the morphisms  in (\ref{j11a}) induce natural morphisms:
\beq\la{j12a}
\xymatrix{
\delta^*_\bullet: {\ph}^{\bullet -1} [-\bullet +1] \iota^* \ar[r] & \iota^* \, {\ph}^{\bullet} [-\bullet], &
\delta^!_\bullet: \iota^! \,{\ph}^{\bullet} [-\bullet]  \ar[r] &  {\ph}^{\bullet+1} [-\bullet -1] \iota^!,
}
\eeq
which are exchanged by Verdier duality.

\end{lm}
\begin{proof}
We prove the lemma for the l.h.s. of (\ref{j11a}). The r.h.s. follows by Verdier duality.

Consider the l.h.s. diagram of (\ref{j11a}), but with the arrows of type $\delta$ and $\e$ removed.
By applying \ci[Prop. 1.1.9, p.23]{bbd}, whose hypotheses are met by virtue of the inequalities (\ref{isrt}) and (\ref{isrt1}),
we see that we can fill in the diagram and make it commutative with unique arrows. The compositions  $\e \delta$ of the vertical arrows give
the desired arrows of type $\gamma$ because in any $t$-structure the structural morphism $\ptd{\bullet -1 } \to {\rm Id}$ factors uniquely
through the structural morphism $\ptd{\bullet } \to {\rm Id}$ via $\gamma_{\leq \bullet}: \ptd{\bullet -1} \to \ptd{\bullet}$.

The morphism on the l.h.s. of  (\ref{j12a}) arises by composing  truncation  applied to $\delta^*$,
with $\delta^*$ applied to the truncation:
\[
\xymatrix{
\ptu{\bullet -1} \ptd{\bullet -1} \iota^*  \ar[rr]^-{\ptu{\bullet -1} \delta^*_{\leq \bullet}}  && \ptu{\bullet -1} \iota^* \ptd{\bullet}
\ar[rr]^-{\delta^*_{\geq \bullet -1} \ptd{\bullet}} && \ptu{\bullet -1} \iota^* \ptd{\bullet}.
}
\]
The morphism on the r.h.s. is obtained in a similar fashion and the verification that the two arrows in (\ref{j12a}) are Verdier dual is left to the reader. 
\end{proof}

\begin{rmk}\la{not1z}
$\;$

\ben
\item
By virtue of  the identity (\ref{tsh}) concerning truncations and shifts,   the morphisms of type $\delta$ in (\ref{j11a})
and (\ref{j12a})  may be re-written as follows:
\beq\la{bnnb}
\xymatrix{
\delta^*_{\leq{\bullet}}: \ptd{\bullet} \iota^*[-1]  \ar[r] & \iota^*[-1] \ptd{\bullet}, &
\delta^!_{\leq{\bullet}}:  \iota^![1]  \ptd{\bullet}  \ar[r] & \ptd{\bullet}  \iota^![1], \\
\delta^*_{\geq{\bullet}}: \ptu{\bullet} \iota^*[-1]  \ar[r] & \iota^*[-1] \ptu{\bullet}, & 
\delta^!_{\geq{\bullet}}:  \iota^! [1]  \ptu{\bullet}  \ar[r] &  \ptu{\bullet}  \iota^![1]\\
\delta^*_\bullet:  {\ph}^{\bullet -1} \iota^*  \ar[r] & \iota^*[-1] {\ph}^{\bullet} , &
\delta^!_\bullet:   \iota^! [1] {\ph}^{\bullet}  \ar[r] & {\ph}^{\bullet +1} \iota^!.
}
\eeq
where Verdier duality exchanges  the morphisms along the two diagonals of the first two rows, and exchanges the two terms
in the third row.

\item
By looking at (\ref{j11a}), and by virtue of Verdier duality,  we deduce that for a given $G\in D(T)$ we have that: 
$\delta^*_{\leq \bullet}(G)$ is an isomorphism
if and only if $\delta^*_{\geq \bullet}(G)$ is an isomorphism if and only if $\delta^!_{\leq \bullet}(G^\vee)$ is an isomorphism if and only if $\delta^!_{\geq \bullet}(G^\vee)$  is an isomorphism.

\item
The same example in Remark \ref{nocvz}.(2) shows that for a given $G\in D(T)$,  having
$\delta^*_{\leq \bullet}(G)$  an isomorphism does not imply that $\delta^!_{\leq \bullet}(G)$ is an isomorphism.

\item
Proposition \ref{iotanoc} gives a criterion for all six morphisms of type $\delta$ to be isomorphisms.
When the morphisms (\ref{bnnb})  are isomorphisms, we may say  that in that case  ``perverse truncation/cohomology commutes up to a suitable shift with restriction
 to the Cartier divisor". While standard truncation commutes with such an un-shifted restriction -in fact with any pull-back-, in general 
perverse truncation does not.

\een
\end{rmk}

\begin{rmk}\la{ecz} In view of  the l.h.s. inequality in (\ref{isrt1}),   the r.h.s. vertex of the distinguished triangle
in the middle row of the l.h.s. of (\ref{j11a})    satisfies the inequality
$\iota^* \ptu{\bullet +1}: D \to {^p\!D}^{\geq \bullet}$. 
By taking the long exact sequence of perverse cohomology sheaves of  said distinguished triangle, the aforementioned inequality yields the
natural isomorphism of functors: 
\beq\la{ecz1}
\xymatrix{
\ptd{\bullet -1} \iota^* \ptd{\bullet} \ar[r]^-\simeq & \ptd{\bullet-1} \iota^*.
}
\eeq
\end{rmk}

\begin{pr}\la{iotanoc} {\rm ({\bf If no constituents,
then the $\delta$'s are isomorphisms})}
Let $\iota: T' \to T$  be as in Lemma {\rm \ref{iotatex}}. 
If $G \in D(T)$ has no constituent supported on $T'$, then the morphisms of type $\delta$  in  (\ref{bnnb})  are isomorphisms.
\end{pr}
\begin{proof}
This is  \ci[Proposition 3.1.4]{dema}. For the reader's convenience, we include it here.

We prove the conclusion for  $\delta^*_{\leq \bullet}(G)$. 
We have that  $\ptd{\bullet} G \in {^p\!D}^{\leq \bullet}$, so that, by (\ref{isrt}), we have that  $\iota^*\ptd{\bullet} G \in {^p\!D}^{\leq \bullet}$, and then, clearly, we have that:
 \beq\la{0129}
 \iota^*\ptd{\bullet} G= \ptd{\bullet} \iota^*\ptd{\bullet} G.
 \eeq
 
 In view of (\ref{0129}) and of (\ref{ecz1}), and by considering the truncation triangle
 $\ptd{\bullet -1} \to \ptd{\bullet} \to {\ph^\bullet}[-\bullet] \rightsquigarrow$ applied to $\iota^*\ptd{\bullet} G$,
 in order to prove the desired conclusion, it is necessary and sufficient  to show that
 $\ph^{\bullet}(\iota^* \ptd{\bullet} G) =0$. 
 
 This can be argued as follows. By taking the long exact sequence of perverse cohomology of the  distinguished triangle
 $\iota^* \ptd{\bullet -1}  G \to \iota^* \ptd{\bullet} G \to \iota^* {\ph^{\bullet}} G [-\bullet]  \rightsquigarrow$, we see
 that it is necessary and sufficient to  show that $\iota^* {\ph^{\bullet}} G [-\bullet] \in {^p\!D}^{\leq \bullet -1}$, or, equivalently, that
 $\iota^* {\ph^{\bullet}} G  \in {^p\!D}^{\leq -1}.$
 
By \ci[4.1.10.ii]{bbd}, we have the distinguished triangle  ${\ph^{-1}} (\iota^* {\ph^{\bullet}} G) [1]
\to \iota^* {\ph^{\bullet}} G \to {\ph^{0}} (\iota^* {\ph^{\bullet}} G)   \rightsquigarrow$ and an epimorphism
$\ph^{\bullet} G \to {\ph^{0}} (\iota^* {\ph^{\bullet}} G)$. Since $G$  is assumed to not have constituents supported at
$Y_s$, we must have $\ph^{0} (\iota^* {\ph^{\bullet}} G)=0$, so that $\iota^* {\ph^{\bullet}} G [-\bullet] \in {^p\!D}^{\leq \bullet -1}$,
as requested. 

We have proved that if $G$ has no constituents supported on $T'$, then $\delta^*_{\leq \bullet}(G)$ is an isomorphism.
By Remark \ref{not1z}, we have that $\delta^*_{\geq \bullet} (G)$ is an isomorphism as well. 

Note that $G$ has no constituents supported on $T'$
if and only if the same is true for  $\ptd{\bullet} G$. We thus have that $\delta^*_{\leq \bullet}(\ptd{\bullet} G)$
and $\delta^*_{\geq \bullet}(\ptd{\bullet} G)$ are isomorphisms.

Now, $\delta^*_\bullet$ is the composition of two isomorphisms, namely $\ptu{\bullet} (\delta_{\leq \bullet}^* (G))$,
followed by $\delta^*_{\geq \bullet} (\ptd{\bullet} G)$.  

We have proved that the morphisms of type $\delta^*$ are isomorphisms.

Note that $G$ has no constituents supported on $T'$
if and only if the same is true for $G^\vee$. By Remark \ref{not1z}.(2), it follows that the morphisms of type $\delta^*$ for $G^\vee$ are isomorphisms,
so that so are the morphisms of type $\delta^!$ for $G$. 

The proposition is proved.

Alternatively, one could also retrace the proof given in detail above for the morphisms of type $\delta^*$
and give a proof for the morphisms of type $\delta^!$.
\end{proof}

Proposition \ref{iotanoc} can be informally summarized as follows: under special circumstances, the perverse truncation and restriction to a  ``non-special" Cartier divisor
commute with each other. One can also say that we can slide one functor across the other.

\subsection{The morphisms \texorpdfstring{$\delta$}{d} and vanishing/nearby cycles}\la{cdtz} $\;$

Lemma \ref{j11l1} is about a Cartier divisor on a variety. Let us specialize to the case when the Cartier divisor
is the special fiber in the context of the vanishing and nearby cycle functors as in \S\ref{vncf}.
In particular, we work over a nonsingular connect curve $S$ and we fix a point $s\in S$.

\begin{lm}\la{kij} {\rm ({\bf  Morphisms of type $\delta$ for $\psi$ and $\phi$})}$\;$

\ben
\item
We have natural morphisms of distinguished triangles:
\beq\la{j11aa}
\xymatrix{
\ptd{\bullet } \psi [-1]   \ar@/_3pc/[dd]|-{\gamma^\psi_{\leq \bullet}}  \ar[r] \ar[d]^-{\delta^\psi_{\leq \bullet}}_-\simeq & \psi[-1] \ar[d]^-= \ar[r] & \ptu{\bullet +1} \psi[-1] 
\ar[d]^{\delta^\psi_{\geq \bullet}}_-\simeq  \ar@/^3pc/[dd]|-{\gamma^\psi_{\geq \bullet}}  \ar@{~>}[r]  & &
\ptd{\bullet} \phi  \ar@/_3pc/[dd]|-{\gamma^\phi_{\leq \bullet}}  \ar[r] \ar[d]^-{\delta^\phi_{\leq \bullet}}_-\simeq & \phi \ar[d]^-= \ar[r] & \ptu{\bullet}  \phi
\ar[d]^{\delta^\phi_{\geq \bullet}}_-\simeq  \ar@/^3pc/[dd]|-{\gamma^\phi_{\geq \bullet}}  \ar@{~>}[r]  &
\\
\psi [-1] \ptd{\bullet}   \ar[r] \ar[d]^-{\e^\psi_{\leq \bullet}} & \psi[-1] \ar[u] \ar[d]^-= \ar[r] &\psi [-1]  \ptu{\bullet +1}    \ar[d]^-{\e^\psi_{\geq \bullet}}  \ar@{~>}[r]  & &
\phi \ptd{\bullet}   \ar[r] \ar[d]^-{\e^\phi_{\leq \bullet}} & \phi \ar[u] \ar[d]^-= \ar[r] &\phi   \ptu{\bullet}    \ar[d]^-{\e^\phi_{\geq \bullet}}  \ar@{~>}[r] &
\\
\ptd{\bullet +1} \psi [-1]  \ar[r]  & \psi [-1]  \ar[r] \ar[u] & \ptu{\bullet +2}  \psi[-1]  \ar@{~>}[r] &, &
\ptd{\bullet +1} \phi   \ar[r]  & \phi  \ar[r] \ar[u] & \ptu{\bullet +2}  \phi  \ar@{~>}[r] &,
}
\eeq
where each diagram is self-dual.

\item
If we denote by $\underline{i^*[-1]}$ and  $\underline{i^![1]}$ the two diagrams (\ref{j11a}) (after having replaced $i^*$ with $i^* [-1]$
and $i^!$ with $i^![1]$; see Remark \ref{not1z}.(1)), and if we denote by $\underline{\psi [-1]}$ and  by $\underline{\phi}$
the  two diagrams in 
(\ref{j11aa}), then there are morphism of diagrams $\underline{i^*[-1]} \to \underline{\psi [-1]} \to \underline{i^![1]}$, in the sense
that all the corresponding square diagrams  are commutative. In particular, we have the commutative diagram:
\beq\la{fr1z}
\xymatrix{
i^*[-1] \ptd{\bullet} \ar[r]  & \psi [-1] \ptd{\bullet} \ar[d]^-{\delta^\psi_{\leq{\bullet}}}_-\simeq
\ar[r] & i^![1]  \ptd{\bullet}  \ar[d]_-{\delta^!_{\leq{\bullet}}}
 \\
 \ptd{\bullet} i^*[-1] \ar[r]  \ar[u]^-{\delta^*_{\leq{\bullet}}} &  \ptd{\bullet} \psi [-1] \ar[r]  \ar[u] &   \ptd{\bullet} i^![1] .
}
\eeq
Similarly, for $\ptu{\bullet}$ and $\ph^{\bullet}$. The resulting diagrams are exchanged -in the case of truncations-
or preserved -in the case of perverse cohomology- by Verdier Duality.
\item
Let $G\in D^b_c(Y)$ be such that $\phi G=0$. 
Then all the vertical arrows in (\ref{fr1z}),
and in its variants involving $\ptu{\bullet}$ and $\ph^{\bullet}$,   evaluated  at $G,$ are isomorphisms.
 In particular, all the arrows of type $\delta$ evaluated  at $G$ in (\ref{bnnb}) are isomorphisms.

\item 
Let $G\in D^b_c(Y)$ be such that $G$ has no constituents supported on $Y_s$. Then all the vertical arrows in (\ref{fr1z}),
and in its variants involving $\ptu{\bullet}$ and $\ph^{\bullet}$,   evaluated  at $G$ are isomorphisms.
In particular, all the arrows of type $\delta$ evaluated  at $G$ in (\ref{bnnb}) are isomorphisms.

\een
\end{lm}
\begin{proof}
(1)
The proof  that there are the natural diagrams (\ref{j11aa}) is the same as the one of Lemma \ref{j11l1}. The fact that the arrows
of type $\delta^\psi$ and $\delta^\phi$ in (\ref{j11aa}) are isomorphisms simply translates the $\frak t$-exactness of $\psi [-1]$ and $\phi$.

(2) Once we note that the dual to $i^*$ is $i^!$, and  that $\psi[-1]$ and $\phi$ are self-dual, hence the self-duality of (\ref{j11aa}), there is nothing left to prove.

(3)
That the horizontal arrows in  (\ref{fr1z}) are isomorphisms   is an immediate consequence of the distinguished triangles (\ref{1p}) and the hypothesis $\phi G=0$. Once the horizontal arrows are isomorphisms, since one of the vertical arrows
is an isomorphism, so are the remaining vertical ones. 

(4) Follows from Proposition \ref{iotanoc}.
\end{proof}

\begin{rmk}\la{nocvz} $\;$

\ben
\item
We have  the following  two implications: a) if $\phi (G)=0$, then $G$ has no constituents on the special fiber
(Lemma \ref{silm});
b) if $G$ has no constituents on the special fiber, then  the morphisms of type $\delta$ in
(\ref{fr1z}) are isomorphisms.
These two implications are combined in Lemma \ref{kij}.(3), which is an amplified version of 
 \ci[Proposition 3.1.5]{dema}. 

\item
 Both the implication a) and b) directly above, as well the implication  in Lemma \ref{kij}.(3), cannot be reversed. For a), see the counterexample in Remark \ref{gt5}. For 2),  as well as for Lemma \ref{kij}.(3), consider the following counterexample:
 take $\iota$ to be the embedding of the origin $s$ into a disk $\disk$, and  $G$ to be $j_! \rat[1] \in P(\disk)$.
Then $G$ admits $\rat_s$ as a constituent and $\phi G=0$. On the other hand $i^* j_!=0$, so that the morphisms of type 
$\delta^*$
are trivially isomorphisms of zero objects.

\item The fact that $G$ has no constituents supported on $Y_s$ does not prevent the following phenomenon:
the number of constituents of $t^* G$ may be strictly smaller than the number of constituents of $i^*G$.
Consider a family $X\to Y \to S$, where $X/S$ is a family of K3 surfaces fibered onto curves, where the general member 
has irreducible fibers, while a special member has some reducible fibers. Then $G:= f_* \rat$ will exhibit this kind of behaviour.
I thank Giulia Sacc\`a,  Antonio Rapagnetta and Junliang Shen  to help me sort this out.
\een
\end{rmk}

\subsection{Relative hard Lefschetz and morphisms of type \texorpdfstring{$\delta$}{d}}\la{vgtr}$\;$

The goal of this section is to prove Proposition \ref{crux} which    is a slight generalization of
\ci[Corollary 3.8]{mish}, and  yields another criterion, this time  involving the Relative Hard Lefschetz Theorem (RHL), for  the morphisms of type $\delta$ to be isomorphisms. 
Proposition \ref{crux} is used in the proofs of Theorems \ref{crit} and \ref{sonoloro}.

Proposition \ref{crux}.(2) is the special case of Proposition \ref{crux}.(1) when the Cartier divisor is the special fiber
in the context of the vanishing/nearby cycle functors. 
We thus start with a variety $X$ and a complex  $F\in P(X) $ perverse and semisimple such that it stays perverse
and semisimple  after restriction  and shift by $[1]$ to the pre-image $X'$  on $X$ of a  Cartier divisor $Y'$ on  $Y.$ This condition on the restriction
should be thought as some kind of weakened transversality  condition of the divisor  $Y'$ on $Y,$ with respect to the morphism, and to the semisimple perverse sheaf $F.$

\begin{rmk}\la{sonomet}
{\rm
The hypotheses of the upcoming  Proposition \ref{crux}  are met, for example,   when  the morphism $f$ is projective,  the varieties $X$ and $X'$ are irreducible orbifolds
and $F=\rat_X [\dim{X}]$. Note also that in this section, we do not  assume that $X,Y$, etc., are varieties over $S$.
}
\end{rmk}

\begin{pr}\la{crux}
{\rm ({\bf RHL criterion for the  morphisms  $\delta$ being isomorphisms)}}

\ben
\item
Let $f:X\to Y$ be a projective morphism of varieties. Let $\iota: Y' \to Y$ be a closed embedding
of pure codimension one such that the resulting open embedding $Y \setminus Y' \to Y$ is affine and 
let:
\beq\la{theta}
\xymatrix{
X' \ar[d]^-f \ar[r]^-\iota & X \ar[d]^-f 
\\
Y' \ar[r]^-\iota & Y
}
\eeq
be the resulting Cartesian diagram.
Let  $F \in D^b_c(X)$  be perverse semisimple and such that $\iota^!F[1] \in D(X')$ (resp. $\iota^* F[-1]$)  is perverse semisimple.

Then we have the following identities: 
\beq\la{lzeq}
 \ptd{\bullet} \left( f_* \iota^! F[1] \right) = \left( \iota^! \ptd{\bullet} f_* F\right) [1] ,
 \quad 
  (resp. \; \ptd{\bullet} \left( f_* \iota^* F[-1] \right) = \left( \iota^* \ptd{\bullet} f_* F\right) [-1]), 
\eeq
and similarly   if $\ptd{\bullet}$  is replaced by $\ptu{\bullet},$ or by $ {\ph}^\bullet$. 

The  complex $f_* F$ has  no constituent supported on
$Y'$.

The morphisms of type  $\delta$ in  (\ref{bnnb})
associated with $f_*F$ are isomorphisms.

\item
Assume, in addition, that $f$ is an $S$-morphism of $S$-varieties,  where $\SP$ is a nonsingular curve, that $s \in S$ is a point and that $Y'=Y_s$.

Then  the morphisms of type  $\delta$ in  (\ref{bnnb}), (\ref{j11aa}) and (\ref{fr1z})
associated with $f_*F$ are isomorphisms.

The complex $f_* F$ has  no constituent supported on
the special fiber $Y_s$.

\een
\end{pr}

\begin{proof}
Since part (2) is a special case of part (1), it is enough to prove part (1). We do so in the case of $\iota^! F[1]$; the case 
of $\iota^* F[-1]$ can be proved in the same way.

We prove the identities (\ref{lzeq})   for $\iota^!$. 
By the Decomposition and Relative Hard Lefschetz Theorems \ref{dt}, we have isomorphisms:
\beq\la{tsplit0}
\xymatrix{
f_* F \ar[r]^-\cong & \bigoplus_{\bullet} {\ph}^{\bullet} (f_*F)[-\bullet], &
f_* \iota^!F[1]  \ar[r]^-\cong & \bigoplus_{\bullet} {\ph}^{\bullet} (f_*\iota^!F[1])[-\bullet],
}
\eeq
\beq\la{rhlp0}
\xymatrix{
\ph^{-\bullet}   f_* F \cong {\ph^{\bullet}} f_* F , & \ph^{-\bullet} (f_* \iota^!F[1])  \cong {\ph^{\bullet}} (f_* \iota^!F[1]).
}
 \eeq

By the base change identity $f_*\iota^! F[1] =\iota^!f_* F[1]$,
and by  the splitting assumptions (\ref{tsplit0}), we get an isomorphism:
\beq\la{111az}
\bigoplus_k \iota^! \left(\ph^k \left( f_*F_X\right) [1]\right) [-k] \cong \bigoplus_k  
\left(\ph^k\left(f_*\iota^!F [1]\right)\right)[-k].
\eeq
At this juncture, it does not seem a priori clear that the desired conclusion follows by taking
cohomology sheaves on both sides, the reason being that the unshifted summands on the  lhs 
are not immediately seen to be perverse. What it is a priori clear is that, due to the shift $[1]$, they live in $^p\!D^{[-1,0]}(Y')$ 
(cf. [BBD, 4.1.10.ii, p.106]).

{\em CLAIM}: the complex  $\iota^! {\ph}^k(f_*F_X)[1]$ is a perverse  sheaf, $\forall k$.

{\em Proof of the CLAIM}. We need to show that ${\ph}^{-1}(\iota^! {\ph}^k(f_*F_X)[1])=0$, $\forall k$. To simplify the notation, we set:
\beq\la{asw}
P^k_{X'}:= {\ph}^k(f_*\iota^!F[1]), \quad Q^k_X:= \iota^! {\ph}^k(f_*F_X) [1] \in {^p\!D}^{[-1,0]}(W)  , \quad
Q^{k,l}_X =  {\ph}^l (Q^k_X), \; l=-1,0.
\eeq
We need to show that $Q^{k,-1}_X=0$, $\forall k$.

We know that   ${\ph}^{-k} \cong {\ph}^k$, $\forall k$, both for $f_*F$ and  for$f_* \iota^!F[1]$, 
so that we have that:
\beq\la{cvd1}
P^{-k}_{X'}= P^{k}_{X'},  \qquad Q^{-k,l}_X=Q^{k,l}_X, \quad \forall k, \;\; \forall l=-1,0.
\eeq

We argue by contradiction and assume that the desired conclusion fails for some $k$,
i.e. that there is $k$ such that $Q_X^{k,-1}\neq 0$.

By taking perverse cohomology in   (\ref{111az}), we get an  isomorphism\beq\la{jk}
P_{X'}^k \cong Q^{k,0}_X \oplus Q^{k+1,-1}_X.
\eeq
Let $-k_0$ be the smallest index $-k$  such that $Q^{-k,-1}\neq 0$. By symmetry, $k_0 \geq 0$.
We thus have $P_{X'}^{-k_0-1} = Q_X^{-k_0-1,0} \oplus Q_X^{-k_0,-1} \neq  Q_X^{-k_0-1,0}$.
By symmetry, $P_{X'}^{k_0+1}  \neq  Q_X^{k_0+1,0}$

On the other hand, we have that  $P_{X'}^{k_0+1}  = Q_X^{k_0+1,0} \oplus Q_X^{k_0+2,-1}= Q_X^{k_0+1,0}$,
because  $Q_X^{k_0+2,-1} = Q_X^{-k_0-2,-1}=0$, by the minimality  property of $-k_0$.

{\em This contradiction establishes the CLAIM}.

At this point, (\ref{111az}), joint with the CLAIM, implies, by taking perverse cohomology sheaves,
that (\ref{lzeq}) holds with $``="$ replaced by an $``\cong"$.
In order to establish  (\ref{lzeq}) with $=$ (canonical isomorphism), we simply observe that  now that  we know that the lhs of (\ref{111az})
is a direct sum of shifted perverse sheaves,  we  reach the desired conclusion by  taking perverse cohomology in 
the isomorphism $\iota^! f_* F [1] = f_* \iota^! F[1]$. 

We have thus proved identities  (\ref{lzeq}).   At this point, the two  final assertions of (1) can be seen to be equivalent;
we show that they hold as follows.

The assertion on the lack of constituents supported on the Cartier divisor $Y'$ follows from the fact that any such constituent would have to show up as a direct summand
supported on $Y'$ and, as such, its $\iota^!=\iota^*$ would show up as a non-zero $Q^{k,-1}$ in the proof above,
a contradiction.

The last  assertion of part (1) follows from Proposition \ref{iotanoc}.
\end{proof}

\begin{rmk}\la{baco} 
As the proof shows, one may replace the assumption that $ \iota^!F[1]$ is perverse semisimple, with the assumption
that $f_* \iota^!F[1]$ satisfies the conclusion of the RHL.   Similarly, if we replace $\iota^! F[1]$ with $\iota^* F[-1]$.
This assumption is met if the boundary is a simple normal crossing divisor and we take $F$ to be the intersection complex, which,
by the very definition of snc divisor, is necessarily the
shifted constant sheaf near the divisor;
see Corollary \ref{vbre}. See also the proof of part II of  Theorem \ref{sonoloro}, and Remark \ref{allokvar}.
\end{rmk}

\subsection{RHL and boundary with normal crossing divisors}\la{bncd}$\;$

Let us place ourselves in the situation of Proposition \ref{crux}, except that we do not assume that $\iota^*F[-1]$ is perverse semisimple on the Cartier divisor on  $X'$.  The goal of this section is to prove Lemma \ref{indz}, to the effect that
if we are in a simple normal crossing divisors situation on a nonsingular variety $X$ of dimension $n$, so that $\iota^*\rat_X [n-1]$
is not perverse semisimple on $X'$, then we still have the useful RHL symmetry (cf. Remark \ref{baco}).
For an application, with details left to the reader, see Remark \ref{allokvar}.

Let $X$ be an irreducible variety, let $n$ be its dimension,  and let $Z\subseteq X$ be a simple normal crossing divisor;
in particular, $X$ is nonsingular near $Z$.
Let $\ms I$ be a finite set indexing the irreducible components $Z_i$ of $Z$, let $I=\{i_1, \ldots, i_k\} \subseteq \ms I$
be any subset, let $Z_I:= \cap_{i \in I} Z_i$.
We have the  standard long exact sequence of sheaves on $X$:
\beq\la{rezol}
\xymatrix{
0 \ar[r] & \rat_Z \ar[r]  & \bigoplus_{|I|=1} \rat_{Z_I} \ar[r] &   \ldots \ar[r] & \bigoplus_{|I|=n} \rat_{Z_I} \ar[r] & 0.
}
\eeq
By splicing (\ref{rezol}), and by  shifting in an appropriate manner, we get the system of  short exact sequences of perverse sheaves in $P(X)$:
\beq\la{spre}
\xymatrix{
0\ar[r] &0 \ar[r] & K^{n} [0]  \ar[r]^-\simeq &  \bigoplus_{|I|=n+1} \rat_{Z_I} [ 0] \ar[r] &0, 
\\
0 \ar[r]  & K^{n} [0] \ar[r] &  K^{n-1} [1]  \ar[r] & \ar[r]  \bigoplus_{|I|=n} \rat_{Z_I} [ 1 ] \ar[r] & 0,
\\
&& \ldots &&,
\\
0 \ar[r]  & K^2 [d_Z -2] \ar[r] &  K^1 [d_Z -1]  \ar[r] & \ar[r]  \bigoplus_{|I|=2} \rat_{Z_I} [ d_Z -1 ] \ar[r] & 0,
\\
0 \ar[r]  & K^1 [d_Z -1] \ar[r] &  \rat_Z [d_Z] \ar[r] & \ar[r]  \bigoplus_{|I|=1} \rat_{Z_I} [ d_Z ] \ar[r] & 0.
}
\eeq

Note that $\rat_Z [d_Z]$ is not semisimple in general, even when $X$ is nonsingular. Nevertheless, we have:

\begin{lm}\la{indz} 
Let $f:X\to Y$ be a proper morphism with  
$X$  nonsingular and irreducible. Then 
$f_* \rat_Z[d_Z]$ has the  RHL symmetry,  i.e. cupping with the first Chern class $L$ of an $f$-ample line bundle induces isomorphisms
\beq\la{rlh}
\xymatrix{
L^{\bullet}: {\ph}^{-\bullet} f_* \rat_Z [d_Z] \ar[r]^-\simeq & {\ph}^{\bullet} f_* \rat_Z [d_Z],
}
\eeq
so that the complex    $f_* \rat_Z[d_Z]$ is $\frak{t}$-split (splits as the direct sum of its shifted perverse cohomology sheaves).
Moreover, we have that:
\beq\la{v9v9}
{\ph}^{\bullet} \left( \iota^* {\ph}^{\star} \left( f_* \rat_Z [d_Z]\right) \right)  =0, \quad \forall \bullet \neq -1, \quad \forall \star.
\eeq
\end{lm}
\begin{proof}
For the RHL symmetry:

We use descending induction on $n$: start from the bottom of (\ref{spre}); go up one step at the time; at each step
use the five lemma applied to:
\beq\la{fce11}
\xymatrix{
0 \ar[r] & 
\ph^ {-\bullet} A \ar[r]  \ar[d]^-{L^\bullet} & 
\ph^ {-\bullet} B \ar[r]  \ar[d]^-{L^\bullet} &
\ph^ {-\bullet} C \ar[r]  \ar[d]^-{L^\bullet} &
0
\\
0 \ar[r] &
\ph^ {\bullet} A \ar[r]    & 
\ph^ {\bullet} B \ar[r]   &
\ph^ {\bullet} C \ar[r]  &
0
}
\eeq
For the splitting: use the Deligne-Lefschetz Criterion.

For the proof of the vanishing (\ref{v9v9}), use the same kind of descending induction.
\end{proof}

\begin{cor}\la{vbre}
Let $f:X \to Y$ be a projective morphism of varieties. Consider  Cartesian diagram  like (\ref{theta})
and assume that $X'$ is supported on a simple normal crossing divisor. Let $F = IC_X$.

Then the conclusions of Proposition \ref{crux} hold.
\end{cor}
\begin{proof} We simply observe that $X$ is assumed to be  nonsingular near $X'$, so that, near $X'$, $IC_X$ is the constant sheaf, up to shifts. The shifts are locally constant integers, and this does not affect the arguments. Now, $F$ is semisimple
on $X$. While $i^*F[-1]$ is not semisimple,  Lemma \ref{indz} applies, so that the RHL holds, and we can
repeat verbatim the proof   of Proposition \ref{crux}.
\end{proof}

\section{The specialization morphism as a filtered morphism}\la{spsfi}

 In this section, we initiate a systematic discussion of when the specialization morphism
 is defined and when it is a filtered isomorphism for the perverse Leray filtrations associated with a morphism. The case of the perverse filtration is the   special case when said morphism is the identity.
 
 \S\ref{arce1} revisits the classical specialization morphism $sp^*(-)$ and frames it in a formalism suited to handle it together with perverse truncation in the later sections. The specialization morphism and its filtered counterparts are  not defined for an arbitrary complex. The main goal of this paper is to identify conditions that ensure that  these morphisms are defined, and that when they are defined 
 they are  isomorphisms and,   in the filtered context, filtered isomorphisms. 
   
  \S\ref{sppflz} introduces the  notion   of $P$-filtered specialization morphism $sp^*_P(-)$, where $P$ stands for the perverse filtration. 
 This can be viewed  as a  special case of the perverse Leray specialization morphism seen later in \S\ref{hh1}. However, it seems useful to discuss it separately, without the complication arising from the additional data of a morphism; this is similar to
 what happens with the Grothendieck and Leray spectral sequences/filtrations. 
 
\S\ref{hh1} introduces the  central-to-this-paper notion of (perverse Leray) $P^f$-filtered specialization morphism
$sp^*_{P^f} (-)$. In order to streamline the discussion, we introduce the commutative diagram (\ref{ar1}).
  We then prove Theorem \ref{crit} which  lists criteria for the specialization morphism to be a filtered isomorphism for the perverse Leray filtration; the earlier Propositions \ref{tpo} and  \ref{tbo} are the analogous criteria for the specialization 
  and (perverse) $P$-filtered specialization morphisms $sp^*_P(-)$ and $sp^*_{P^f} (-).$

 \S\ref{spcor} is devoted to give a criterion for the specialization morphism to exist and to be a filtered isomorphism
 when a suitable compactification is available: see Set-up \ref{setzxw} and Proposition \ref{r54}.
  
 In \S\ref{tent}, we continue the theme of \S\ref{spcor} -the existence of a suitable compactification- and
 consider the problem of the resulting long exact sequence of cohomology associated with the compactification
 -boundary, space, open part-. We 
 prove Theorem \ref{sonoloro} which gives  criteria for the three specialization morphisms involved to exist, to be filtered isomorphisms, and to fit into an isomorphism of the resulting long exact sequences for the triples, at the special point $s$   and at the general point $t.$ 
 
 We discuss the relation between the various criteria in Remark \ref{hoi}.
 
 We apply these criteria to the Hitchin morphism in \S\ref{apell}.

 \subsection{The specialization morphism revisited}\la{arce1}$\;$

 Let things be as in \S\ref{vncf}. In particular, we have:  a morphism $v:Y\to S,$ a point $s \in S$, a complex $G \in D^b_c(Y).$

 By combining the base change properties in (\ref{fdtbis}) and (\ref{fd1}), and by considering the morphisms of type $\nu$
 in (\ref{gh}), we obtain the commutative diagram 
 of morphisms of functors $D^b_c(Y) \to D^b_c(pt)$ (cf. Remark \ref{idspt}):
\beq\la{ar-1}
 \xymatrix{
 i^*[-1] v_* \ar[rr]^-{\s^*}  \ar[d]^{{bc}^{i^*v_*}}&&
\psi[-1] v_*   \ar[rr]^-{\s^!} \ar[d]^{{bc}^{\psi v_*}}  \ar[drr]^-(.3){sp^!} &&
i^![1] v_* \ar[d]_-=^-{{bc}^{i^!v_*}}  &
(\phi v_*)
\\
 v_*  i^*[-1]  \ar[rr]^-{\s^*}
 \ar@/0pt/@{-->}[urr]^-(.7){?sp^*?}    \ar@/_2pc/[rrrr]^-{\nu = {^2\!\nu}}
 &&
v_* \psi[-1]    \ar[rr]^-{\s^!}    &&
v_* i^![1]  \ar[u] & (v_* \phi), 
 }
 \eeq
where:   there is a $\nu$ for each of the two rows, i.e. $^1\!\nu$ and $^2\!\nu$,  but we have indicated only the one for the second row; the terms $\phi v_*$ and $v_* \phi$ in parentheses  are, up to shift,  the cones of the morphisms of type $\s$ (cf. Remark \ref{coniveri}).

We have the functors $\nu, sp^!: D^b_c(Y) \to D^b_c(pt)$.
Due to the directions of the arrows on the l.h.s. of  (\ref{ar-1}), the broken arrow ``$?{sp^*}?$'' is not defined as a morphism of functors, but, as it is discussed below,  it can happen to be defined for some $G \in  D^b_c(Y)$.

 \begin{rmk}\la{coniveri} 
 {\rm ({\bf Cones in (\ref{ar-1})})}
By (\ref{1p}), the cones of the morphisms $\s^*$ and $\s^!$ on the top row of (\ref{ar-1}) are $\phi v_*$ and $\phi v_* [1]$, respectively.
Then,  for $G\in D^b_c(Y)$, the morphism $\s^*(v_*G)$  is an isomorphism, if and only the morphism $\s^! (v_*G)$ is an isomorphism,  if and only if $\phi v_* G=0$.
What above remains true for the bottom row, with $v_* \phi$ replacing $\phi v_*,$ etc. 
 \end{rmk}

 \begin{rmk}\la{ffj}
 {\rm
 ({\bf Interrelations between vanishings})
 }
In general, the two conditions $v_* \phi G=0$ and $\phi v_* G=0$ are independent.
 We have the following implications:
 
 \centerline{
$(\phi G=0) \Longrightarrow (v_* \phi G=0);$  $v$ proper $\Longrightarrow$
$((v_*\phi G =0) \Longleftrightarrow (\phi v_* G=0))$.}
\end{rmk}

 Since in general the base change morphism ${bc}^{i^*v_*}$ is  not an  isomorphism, the sought-after morphisms
 $?{sp^*}?$ of functors $D^b_c(Y) \to D^b_c(pt)$ 
 is not defined.

\begin{defi}\la{defsp}
{\rm
({\bf The specialization morphism ${sp^*}(G)$})
}
Let $G\in D^b_c(Y)$.
We say that the specialization morphism ${sp^*}(G)$ is defined for $G$ 
if the base change morphism ${bc}^{i^*v_*}(G)$ is an isomorphism. In this case, we define the morphism
in $D^b_c(pt)$:
\beq\la{spzzx}
\xymatrix{
{sp^*}(G):=  \s^* ({bc}^{i^*v_*})^{-1}    : R\Gamma (Y_s, i^*G) =  v_* i^* G   \ar[r] &
\psi v_* G= R\Gamma (s, \psi v_* G) .
}
\eeq
\end{defi}

When $sp^*(G)$ is defined we have that:
\beq\la{vv}
\xymatrix{
\nu(G)=sp^!(G)\circ sp^*(G).
}
\eeq

\begin{rmk}\la{box}
The part of diagram (\ref{ar-1}) that is needed to define  specialization morphisms is:
\beq\la{boxs}
\xymatrix{
v_* i^*[-1]
&
i^*[-1]v_* \ar[r]^-{\s^*}  \ar[l]_-{bc^{i^*v_*}} 
&
  \psi [-1] v_*.  
}
\eeq
The harmless shift by $[-1]$ is  convenient later on, when dealing with the morphisms of type $\delta;$ see diagram
(\ref{ar0}). The reason for having embedded (\ref{boxs}) in (\ref{ar-1}) is explained in  Remark \ref{kaw3}. 
\end{rmk}

\begin{rmk}\la{f01}$\;$
\ben
\item
 By using (\ref{r401}), we may re-write  (\ref{spzzx}) as follows:
\beq\la{spz1b}
\xymatrix{
{sp^*}(G):  R\Gamma (Y_s, i^*G) \ar[r] 
&
\psi v_* G \ar[r]^-\sim
&
R\Gamma (Y_t, t^*G).
}
\eeq
The ambiguity involved in this re-writing (see Fact \ref{jo}, especially (\ref{r401})) plays no role in this paper: firstly, 
the image lands in the monodromy invariants; secondly, we have limited ourselves to criteria for when the specialization morphism is:  defined; an isomorphism; filtered for the perverse filtrations; a filtered isomorphism for the perverse filtrations.
All of these issues can be settled at a point $t$ general for $G,$   independently of the monodromy action at $t$. 
\item
Even when $sp^* (G)$ is defined,   I  do not see a reason for the base change morphism 
\beq\la{bcniso0}
\xymatrix{
bc^{\psi v_*}:  (R\Gamma (Y_t, t^*G) \simeq) \psi v_* G \ar[r] & v_* \psi G =R\Gamma (Y_s, \psi G)
}
\eeq
to be  necessarily an isomorphism.
\item
In view of the isomorphism $\psi v_* G \simeq R\Gamma (Y_t, t^*G)$, typically, and this paper is no exception, we are interested in $\psi v_* G$, and not as much in $v_* \psi G$. This is important to keep in mind when, later we consider the filtered version of the specialization morphisms
see (\ref{ar0}), (\ref{spz1}) and  (\ref{spz1cbix}) for the perverse version, and see (\ref{ar1}), (\ref{spz1bix}) and 
(\ref{spz1c}) for the perverse Leray version. Of course, the base change morphism $bc^{\psi v_*}:\psi v_*   \to v_*\psi $ is an isomorphism when $v$ is proper.
\een
\end{rmk}

\begin{rmk}\la{ft}
{\rm
 ({\bf $v$ proper})
 }
If $v:Y\to S$ is proper, then, by the proper base change theorem, we get the functor:
 ${sp}^* : D^b_c(Y) \to D({pt})$. 

\end{rmk}

\begin{rmk}\la{spsh} {\rm
 ({\bf The morphism of functors $sp^!: \psi [-1] v_* \to v_* i^![1]$})
 }
Regardless of whether $v$ is proper or not,  the functor $sp^!: D^b_c(Y) \to D^b_c(pt)$ is defined.
In analogy with (\ref{spz1b}),  and taking into account the definition of cohomology with supports
on the closed set $Y_s$, for a given $G\in D^b_c(Y)$, it takes the form of a morphism:
\beq\la{morp sh}
\xymatrix{
sp^!(G) :  R\Gamma (Y_t, t^*G) \ar[r]^-\simeq &  \psi  v_* G  \ar[r] & R\Gamma (Y_s, i^![2] G)= R\Gamma_{Y_s}(Y,G[2]).
}
\eeq
Most of the results proved in this paper for the morphism $sp^*(G)$ (when defined), hold for the morphism 
$sp^!$. 
In fact, the proofs are simpler in this case, since we do not need to establish the existence of $sp^!$ 
along the way, the way we must do for a  $sp^*(G)$.
On the other hand, in general, the r.h.s. of (\ref{morp sh}) is not easily relatable to, say, $R\Gamma (Y_s, i^*G)$.
For this reason, we have not included in this paper the   $sp^!$-versions of the results. 
We simply observe that when, for some necessarily special $G$, the natural morphism $\nu: i^*[-1] (G) \to i^![1](G)$ is an isomorphism, then, we do have that
$sp^!(G) :  R\Gamma (Y_t, t^*G) \to R\Gamma(Y_s,i^* G)$ and this can be a very useful tool, especially, when $sp^*(G)$ fails to be defined. This proved to be crucial 
in the paper \ci{demash2019} (where due to the particular set-up there, we opted for a more direct construction
of the morphism $sp^!$).
\end{rmk}

\begin{rmk}\la{short}
In what follows, ``($\Leftarrow (\phi G=0)$)" is  short for ``(which is implied by $(\phi G=0)$)".
\end{rmk}
\begin{pr}\la{tpo} $\;$
{\rm 
({\bf  Criteria for the existence of ${sp^*}(G)$ and for it being an isomorphism})
}
Let $G \in D^b_c(Y).$ Recall Remark \ref{short}.

\ben
\item If $v$ is proper, then ${sp}^*:D^b_c(Y) \to D({pt})$ is defined as a functor.

If in addition $\phi v_* G=0$ ($\Leftarrow (\phi G=0)$),   then ${sp}^* (G)$ is   an isomorphism. 
\item
If  $\phi v_* G=0$ and  $\phi G=0$, then ${sp}^* (G)$ is defined and  an isomorphism. 
\een
\end{pr}
\begin{proof}

(1). For the first assertion,  see Remark \ref{ft}.
For the second one,  we are left with showing that the morphism $\s^* (v_* G)$ on the top row of (\ref{ar-1})  is an isomorphism. This follows  from the fact that
its cone  $\phi v_* G=0$.

(2). The hypotheses imply that the cones of the four horizontal arrows in (\ref{ar-1}) are zero. 
These four  arrows are then all isomorphisms.
Since one of the three vertical arrows
in (\ref{ar-1})
is an isomorphism, so are the remaining two, and  we are done.
\end{proof}

\begin{rmk}\la{kaw3}
{\rm
 ({\bf Vanishing of  $\phi$ and base change})
 }
 In order to define the specialization morphism, we only need the l.h.s. of diagram (\ref{ar-1}).
The proof of  Proposition \ref{tpo}.(2)  shows  that the two conditions $\phi G=0$ and $\phi v_* G=0$ combined,   imply that the base change morphism
 ${bc}^{i^*v_*} (G)$  for $G$ is an isomorphism. This is the reason why we have introduced diagram  (\ref{ar-1}):
  if, after having plugged some $G \in D^b_c(Y)$,  the four horizontal arrows in (\ref{ar-1}) are isomorphisms, then, because
 of the presence of  base change isomorphism ${bc}^{i^!v_*}$, the three vertical arrows evaluated at $G$ are also isomorphism, i.e. the 
 base change morphism $bc^{i^* v_*}$  is also an isomorphism, and the specialization morphism ${sp^*}(G)$ is defined.
\end{rmk}

\subsection{The specialization morphism and the perverse filtration}\la{sppflz}$\;$

Let things be as in \S\ref{vncf}. In particular, we have:  a   morphisms  $v:Y\to S,$ a point $s \in S$, a complex $G \in D^b_c(Y).$

By analogy to (\ref{ar-1}), and by using the naturality properties of the morphisms of type $\delta$
in \S\ref{prep},
we obtain the commutative diagram of  morphism of systems of functors:
 \beq\la{ar0}
 \xymatrix{
 i^*[-1] v_* \ptd{\bullet} \ar[rr]^-{^1\!\s^*_{\leq \bullet}}  \ar[d]^-{{bc}^{i^*v_*}_{\leq \bullet}}&&
\psi[-1] v_* \ptd{\bullet}  \ar[rr]^-{^1\!\s^!_{\leq \bullet}} \ar[d]^{{bc}^{\psi v_*}}  \ar[ddrr]^-(.3){sp^!_{\leq \bullet}} &&
i^![1] v_* \ptd{\bullet}  \ar[d]_-=^{{bc}^{i^!v_*}_{\leq \bullet}}  & (\phi v_* \ptd{\bullet})
\\
 v_*  i^*[-1] \ptd{\bullet}  \ar[rr]^-(0.4){^2\!\s^*_{\leq \bullet}}  &&
v_* \psi[-1]  \ptd{\bullet}  \ar[rr]^-(0.6){^2\!\s^!_{\leq \bullet}}   \ar[d]^-{\delta^\psi_{\leq \bullet}}_-= &&
v_* i^![1]  \ptd{\bullet}  \ar[d]^-{\delta^!_{\leq \bullet}} \ar[u] &   (v_*\phi \ptd{\bullet})
  \\
v_*  \ptd{\bullet}  i^*[-1]    \ar[rr]^-{^3\!\s^*_{\leq \bullet}}  \ar[u]_-{\delta^*_{\leq \bullet}} 
\ar@/0pt/@{-->}[uurr]^-(.7){?sp^*_{\leq \bullet}?}    \ar@/_2pc/[rrrr]^-{\nu_{\leq \bullet} = \, ^3\!\nu_{\leq \bullet}}   &&
 v_*   \ptd{\bullet} \psi[-1]  \ar[rr]^-{^3\!\s^!_{\leq \bullet}}  \ar[u]     &&
  v_*  \ptd{\bullet}  i^![1], & (\phi).
 }
 \eeq
 
 \begin{rmk}\la{occio} 
 {\rm
 ({\bf Cones/not cones in (\ref{ar0})})
 }
 In  the top two rows of diagram (\ref{ar0}), 
 we have indicated in parentheses  the cones -up to shift- of the morphisms of type $\s$: apply (\ref{ar-1}) and Remark
 \ref{coniveri} to $\ptd{\bullet}$.
 Since truncation is not an exact functor,  the cones of the morphisms of type $\s$ in the third row
 are not shifts of $v_* \ptd{\bullet} \phi$;   the term in parentheses  on the third bottom row
 is a term whose vanishing when evaluated at $G$  ensures that the arrows of type $\s (G)$ on the third row are isomorphisms.
\end{rmk}
 
 The  reason why (\ref{ar0}) has three rows, instead of the two in  (cf. \ref{ar-1}), is that
 $\ptd{\bullet}$ does not commute with $i^*[-1]$ -nor with any other shift-; the morphisms of type $\delta$ measure the failure of this commutativity.

 \begin{rmk}\la{ffj1} 
 {\rm
 ({\bf Interrelations between vanishings})
 }
 The conditions   $\phi v_* \ptd{\bullet} G=0$ and $v_* \phi \ptd{\bullet} G=0$ are independent;
 if $v$ is proper, then they are equivalent.
By the $\frak t$-exactness of $\phi$ and the additivity of $v_*$, we have that:
$v_*\phi \ptd{\bullet}= v_* \ptd{\bullet} \phi$, and that:   $(\phi G=0) \Longleftrightarrow (\phi \ptd{\bullet} G=0) \Longrightarrow v_* \phi  \ptd{\bullet}G=0$. 

\end{rmk}

Since in general the base change morphism ${bc}^{i^*v_*}_{\leq \bullet}$ are not  isomorphisms, the sought-after
morphism of  systems of functors $?{sp^*_{\leq \bullet}}?$ is not defined. On the other hand, there can be a corresponding arrow for special
$G \in D^b_c(Y)$. Recall the notion of  ``systems" (cf. \S\ref{gennot}).

\begin{defi}\la{defspf}
{\rm
({\bf The $P$-filtered specialization morphism ${sp^*_P}(G)$})
}
Let $G\in D^b_c(Y)$.
We say that the $P$-filtered specialization morphism ${sp^*_P}(G)$ is defined for $G$ 
if the base change morphisms ${bc}^{i^*v_*}_{\leq \bullet}(G)$ are isomorphisms, in which case, since 
the inverse isomorphisms form an isomorphism of systems, we can define the arrow in $DF(pt)$:
 ($bc^{i^*v_*}_P$ denotes the  evident
resulting 
filtered isomorphism, and $\delta_P$ the evident resulting filtered morphism)
\beq\la{spz1}
\xymatrix{
{sp^*_P}(G):=  \s^*_P ({bc}^{i^*v_*}_P)^{-1}   \delta^*_P  : v_* i^* G= (R\Gamma (Y_s, i^*G) ,P) \ar[r]
&
(\psi v_* G, P),
}
\eeq
where the last term is the one associated with the system $\psi v_* \ptd{\bullet} G$. 
\end{defi}

We have functors $\nu_P, sp^!_P: D^b_c(Y) \to DF(pt)$. When $sp^*_P(G)$ is defined, we have that:
\beq\la{vvp}
\xymatrix{
\nu_P(G)=sp^!_P(G)\circ sp^*_P(G).
}
\eeq

\begin{rmk}\la{boxfil}
Remark \ref{box} holds essentially verbatim in the context of diagram (\ref{ar0}). In particular, we have:
\beq\la{boxs0}
\xymatrix{
v_* \ptd{\bullet} i^*[-1] \ar[r]^-{\delta^*_{\leq \bullet}}
&
v_* i^*[-1] \ptd{\bullet}
&
i^*[-1]v_* \ptd{\bullet} \ar[r]^-{\s^*}  \ar[l]_-{bc^{i^*v_*}} 
&
  \psi [-1] v_* \ptd{\bullet}.  
}
\eeq
\end{rmk}

\begin{rmk}\la{f02}$\;$

\ben
\item
 By using (\ref{r403}), the filtered analogue of   (\ref{spzzx}) takes the following form:
\beq\la{spz1cbix}
\xymatrix{
{sp^*_P}(G):  (R\Gamma (Y_s, i^*G),P) \ar[r]&
( \psi  v_* G,P) \ar[r]^-\sim
&
(R\Gamma (Y_t, t^*G),P).
}
\eeq
The  analogue of  Remark \ref{f01} on the  ambiguities due to monodromy holds in this context. 
\item
Even when $sp^* (G)$ is defined,  the base change morphism: 
\beq\la{bcniso}
\xymatrix{
bc^{\psi v_*}:  (R\Gamma (Y_t, t^*G) \simeq) \psi v_* G \ar[r] & v_* \psi G =R\Gamma (Y_s, \psi G)
}
\eeq
 is not necessarily an isomorphism. 
 In particular, the filtration $P$ in $(\psi v_* G, P)$ is not  the perverse  filtration
 for a complex on $Y_s$.
\een
\end{rmk}

\begin{rmk}\la{ft1}
{\rm
 ({\bf $v$ proper})
 }
The evident analogue of Remark \ref{ft} holds in this perverse filtered context:
if $v:Y\to S$ is proper, then, by the proper base change theorem, we get the functor:
 ${sp}^*_P : D^b_c(Y) \to DF({pt}).$
\end{rmk}

Proposition \ref{tbo} below is the $P$-filtered counterpart to the  un-filtered Proposition \ref{tpo}.

\begin{pr}\la{tbo}
{\rm 
({\bf  Criteria for the existence of ${sp^*_P}(G)$ and for it being a filtered isomorphism})
}
$\;$

\ben

\item\la{tbo1}
If $v$ is proper, then ${sp^*_P}: D^b_c(Y) \to DF({pt})$ is defined as a functor.

If in addition $\phi G=0$, then  ${sp^*_P}(G)$ is a  filtered isomorphism.
\item\la{tbo4}
If $\phi v_* \ptd{\bullet} G=0$ and $v_* \phi \ptd{\bullet} G=0$, then  ${sp^*_P}(G)$ is defined.

If in addition $\phi G=0$, then  ${sp^*_P}(G)$ is a  filtered isomorphism.

\item\la{tbo5} 
If $G$ is semisimple, $\phi v_*G=0$ and $v_* \phi  G=0$,  then  ${sp^*_P}(G)$ is defined.

If in addition $\phi G=0$, then  ${sp^*_P}(G)$ is a  filtered isomorphism.

\een

\end{pr}
\begin{proof}

(1). For the first assertion, see Remark \ref{ft1}. If we assume that $\phi G=0$, then, in view of Remarks
\ref{occio} and \ref{ffj1}, all horizontal arrows in the commutative diagram (\ref{ar0}) are isomorphisms. Since each row of vertical arrows contains an isomorphisms, all arrows in (\ref{ar0}) are isomorphisms, so that
${sp^*_P}$ is a filtered isomorphism.

(2). The two vanishing assumptions imply that all the four horizontal arrows  in the first two rows of (\ref{ar0}) are isomorphisms. As seen above, this implies that the vertical base change arrows are isomorphisms as well, 
so that ${sp^*_P}$ is defined.
If we we assume that $\phi G=0$, then we conclude as above that ${sp^*_P}$ is a filtered isomorphism.

(3). Since $G$ is semisimple, we have that:  $\phi v_* G=0$ is equivalent to $\phi  v_* \ptd{\bullet} G=0$;
$v_* \phi G=0$ is equivalent to $  v_* \phi \ptd{\bullet} G=0$, Now (3) follows from (2).
\end{proof}

\begin{rmk}\la{kaw4}
{\rm
 ({\bf Vanishing of $\phi$ and base change})
 }
The analogue of  Remark \ref{kaw3}, on the vanishing of $\phi$ assumptions implying   base change isomorphisms, holds in the context of the proof of Proposition \ref{tbo}.  
\end{rmk}

\begin{??}\la{quno} 
I do not know of an example where:  
$v$  is  proper -so that ${sp^*}$ and  $sp^*_{P}$ are defined-, where $\phi v_* G=0$
-so that ${sp}(G)$ is an isomorphism-,  but  where ${sp_P}(G)$  is not an  isomorphism (i.e. in the filtered sense). The issue is that
while $\delta^*_{\leq \bullet}(G)$ in (\ref{ar0}) is an isomorphism for $\bullet \gg 0$, it is not clear to me what happens for other values of 
$\bullet$.  We know that if in addition $\phi G=0$, then $sp^*_P$ is an isomorphism.
\end{??}

\subsection{The specialization morphism and the perverse Leray filtration}\la{hh1} $\;$

Let things be  as in \S\ref{vncf}.  In particular, we have:   morphisms $f: X\to Y$ and $v:Y\to S,$ a point 
$s \in S$, a complex $G:=f_*F \in D^b_c(Y)$ for $F \in D^b_c(X)$.

By analogy to (\ref{ar-1}) and (\ref{ar0}), and by using the naturality properties of the morphisms of type $\delta$
in \S\ref{prep},
we obtain the commutative diagram of  morphisms of  systems of functors:
\beq\la{ar1}
\xymatrix{
i^*[-1] v_* \ptd{\bullet} f_* \ar[r]^-{^1\!\s^*_{\leq \bullet}}  \ar[d]_-{{bc}^{i^*v_*}_{\leq \bullet}} & \psi [-1] v_* \ptd{\bullet} f_*  \ar[d]^-4  
 \ar[r]^-{^1\!\s^!_{\leq \bullet}}  
\ar[dddr]^-{sp^!_{P^f,\leq \bullet}}
& i^![1] v_* \ptd{\bullet} f_* \ar[d]_-=^-{{bc}^{i^!v_*}_{\leq \bullet}} & ((\phi v_* \ptd{\bullet} f_* F=0) \Leftrightarrow {^1\!\s} \,{\rm isos})
\\
 v_* i^*[-1] \ptd{\bullet} f_* \ar[r]^-{^2\!\s^*_{\leq \bullet}}    &  v_*  \psi [-1]  \ptd{\bullet} f_*    \ar[r]^{^2\!\s^!_{\leq \bullet}}  \ar[d]_-= &  v_*   i^![1] \ptd{\bullet} f_* 
 \ar[d]^-{\delta^!_{\leq \bullet}}  \ar[u] & ((v_* \phi  \ptd{\bullet} f_* F=0) \Leftrightarrow {^2\!\s} \,{\rm isos})
\\
v_*  \ptd{\bullet} i^*[-1]  f_* \ar[r]^-{^3\!\s^*_{\leq \bullet}}  \ar[u]^-{\delta^*_{\leq \bullet}} \ar[d]_-{{bc}^{i^*f_*}_{\leq \bullet}}   &  v_*  \ptd{\bullet}  \psi [-1]  f_*    \ar[r]^{^3\!\s^!_{\leq \bullet}} \ar[u]^-=
\ar[d]^-{4'}   &  v_*    \ptd{\bullet} i^![1]  f_*  \ar[d]^-{{bc}^{i^!f_*}_{\leq \bullet}}_-= & ((\phi  f_* F=0) \Rightarrow 
{^3\!\s} \,{\rm isos})
\\
  v_*  \ptd{\bullet}  f_* i^*[-1]   \ar[r]^-{^4\!\s^*_{\leq \bullet}}  \ar@/0pt/@{-->}[uuur]^-{?sp^*_{P^f,\leq \bullet}?}  
  \ar@/_2pc/[rr]^-{{\nu_{\leq \bullet} = \, ^4\!\nu_{\leq \bullet}}}
  &  v_*  \ptd{\bullet}    f_* \psi [-1]    \ar[r]^{^4\!\s^!_{\leq \bullet}}  &  
  v_*    \ptd{\bullet}   f_* i^![1] \ar[u] 
  & ((f_*\phi F=0) \Rightarrow {^4\!\s} \,{\rm isos}).
}
\eeq

\begin{rmk}\la{occiz} {\rm
 ({\bf Cones/not cones in (\ref{ar1})})
 }
 This remark is analogous to Remark \ref{occio}.
The cones of the morphisms of type $\s$ appearing on the first two rows are indicated in parentheses on the r.h.s.
The terms in parentheses on the third and fourth row are not the cones of the morphisms of type $\s$ (truncation is not an exact functor), but their vanishing implies that the morphisms of type $\s$ are isomorphisms.
\end{rmk}

 \begin{rmk}\la{ffj2} 
 {\rm
 ({\bf Interrelations between vanishings})
 }
 The conditions $\phi F =0$ and $\phi f_* F=0$ are independent; if $f$ is proper, then $\phi F=0$ implies $\phi f_* F=f_* \phi F=0$.
Note that $\phi f_* F=0$ is equivalent to $\phi \ptd{\bullet} f_* F=0$, and it implies  $v_* \phi \ptd{\bullet} f_* F=0$. 

The conditions $v_* \phi \ptd{\bullet} f_* F=0$ and $\phi v_* \ptd{\bullet} f_* F=0$ are independent;
if $v$ is proper, then they are equivalent.
\end{rmk}

\begin{defi}\la{defspf1}
{\rm
({\bf The $P^f$-filtered specialization morphism ${sp^*_{P^f}}(F)$})
}
We say that the $P^f$-filtered specialization morphism ${sp^*_{P^f}}(F)$ is defined for $F \in D^b_c(X)$ 
if the base change morphisms ${bc}^{i^*v_*}_{\leq \bullet}(F)$  and 
${bc}^{i^*f_*}_{\leq \bullet}(F)$
are isomorphisms. In this case,  since 
the inverse isomorphisms form a morphism of systems, we can define the arrow in $DF(pt)$:
\beq\la{spz1bix}
\xymatrix{
{sp^*_{P^f}}(F):= {^1\!\s^*_{P}} (bc_{P}^{i^*v_*})^{-1} \delta^*_{P^f} (bc_{P}^{i^*f_*})^{-1}
   : v_* i^* G= (R\Gamma (X_s, i^*F) ,P^f) \ar[r] &
(\psi v_* f_* F, P),
}
\eeq
where the last term is the one associated with the system $\psi v_* \ptd{\bullet} f_* F.$
\end{defi}

We have functors $\nu_{P^f}, sp^!_{P^f}: D^b_c(X) \to DF(pt)$. When $sp^*_{P^f}(F)$ is defined, we have that:
\beq\la{vvpf}
\xymatrix{
\nu_{P^f}(F)=sp^!_{P^f}(F)\circ sp^*_{P^f}(F).
}
\eeq

\begin{rmk}\la{boxfilp}
Remark \ref{box} holds essentially verbatim in the context of diagram (\ref{ar1}). In particular, we have:
\beq\la{boxs1}
\xymatrix{
v \ptd{\bullet} f_* i^* [-1] 
&
v_* \ptd{\bullet} i^*[-1] f_* \ar[r]^-{\delta^*_{\leq \bullet}} \ar[l]_-{bc_{\leq{\bullet}}^{i^*f_*}}
&
v_* i^*[-1] \ptd{\bullet} f_*
&
i^*[-1]v_* \ptd{\bullet} f_*  \ar[d]^-{\s^*}  \ar[l]_-{bc^{i^*v_*}} \\
&&&
  \psi [-1] v_* \ptd{\bullet} f_*.  
}
\eeq

\end{rmk}

\begin{rmk}\la{f03}
 By using (\ref{r405}), we may re-write  (\ref{spz1bix}) as follows:
\beq\la{spz1c}
\xymatrix{
{sp^*_{P^f}}(F):  (R\Gamma (X_s, i^*F),P^{f_s}) \ar[r]   &
(\psi v_* f_*F, P) \ar[r]^-\sim &
(R\Gamma (X_t, t^*F),P^{f_t}).
}
\eeq
The  analogue of  Remark \ref{f02}, on the  ambiguities due to monodromy and on the meaning of $P$
in $(\psi v_* f_* F, P)$, holds in this context. 
\end{rmk}

\begin{rmk}\la{ft2}
{\rm
 ({\bf $v$ proper})
 }
The evident analogue of Remarks \ref{ft} and \ref{ft1} holds in this perverse Leray filtered context: 
if $f$ and $v$ are proper, then we get the functor:
 ${sp}^*_{P^f} : D^b_c(X) \to DF({pt}).$
\end{rmk}

The following  theorem is a perverse Leray analogue of Proposition \ref{tbo}.
Recall the notational  Remark  \ref{short}.

\begin{tm}\la{crit}$\;$
{\rm ({\bf Criteria for the existence of ${sp^*_{P^f}}(F)$ and for it being a filtered isomorphism})}

\ben
\item
If 
$v$ and $f$ are  proper, then   ${sp^*_{P^f}}: D^b_c(X) \to DF({pt})$ is a functor.

If in addition $\phi f_* F=0$ ($\Leftarrow (\phi F=0)$)  then ${sp^*_{P^f}}(F)$ is defined and a filtered isomorphism.

\item
If $v$ proper, $\phi f_* F=0$, and $f_* \phi F=0$ ($\Leftarrow (\phi F=0)$),  then ${sp^*_{P^f}}(F)$ is defined and a  filtered isomorphism.

\item
If $f$ is proper, $\phi v_*\ptd{\bullet} f_*F=0,$ and $v_*\phi \ptd{\bullet} f_*F=0$
($\Leftarrow (\phi f_* F=0) \Leftarrow (\phi F=0)$), then ${sp^*_{P^f}}(F)$ is defined.

If in addition $\phi f_* F=0$ ($\Leftarrow (\phi F=0)$),  then ${sp^*_{P^f}}(F)$ is defined and a filtered isomorphism.

\item
If $\phi v_* \ptd{\bullet} f_* F=0$,   $\phi f_* F=0,$ and $f_* \phi F=0$ ($\Leftarrow (\phi F=0)$), then 
${sp^*_{P^f}}(F)$ is defined and a filtered isomorphism.

\item 
 If $f$ is proper and $F$ is semisimple, $\phi v_* f_* F=0$ and $v_* \phi f_* F=0$
($\Leftarrow (\phi f_* F=0) \Leftarrow (\phi F=0)$),
 then ${sp^*_{P^f}}(F)$ is defined.
 
 If in addition, $\phi f_*F=0$ ($\Leftarrow (\phi F=0)$), then ${sp^*_{P^f}}(F)$ is a filtered isomorphism.

\item
If $v$ is proper, $f$ is projective and $F$ and $i^*F[-1]$ are semisimple, and $\phi v_* f_* F=0$
($\Leftarrow (\phi f_* F=0)    \Leftarrow (\phi F=0)$), then ${sp^*_{P^f}}(F)$ is defined and a filtered isomorphism.

\item
If $f$ is projective and $F$ and $i^*F[-1]$ are semisimple,  $\phi v_* f_* F=0$ and $ v_* \phi f_* F=0$ 
($\Leftarrow (\phi f_* F=0) \Leftarrow (\phi F=0)$,
then ${sp^*_{P^f}}(F)$ is defined and a filtered isomorphism.
\een
\end{tm}

\begin{proof}
(1). For the first assertion in (1) , see Remark \ref{ft2}. We now prove the second assertion in (1). Keep in mind that the base change morphisms are all isomorphisms.
Since $f$ is proper,  $\phi f_* F=0$ implies that -it is equivalent to-
 $f_* \phi F=0$, as well as $v_* \phi \ptd{\bullet} f_* F=0$.  It follows that the four horizontal arrows on row two and three in (\ref{ar1}) are isomorphisms.
 This implies that so are the vertical arrows of type $\delta$.  It remains to show that the morphisms
 ${^1\!\s^*_{\leq \bullet}}$ are isomorphisms,
 which follows from $ \phi  v_* \ptd{\bullet} f_* F= v_* \phi \ptd{\bullet} f_* F=0$.

(2). Since $v$ is proper, rows one and two are identified. The vanishing assumptions imply that all seven arrows on rows three and four are isomorphisms. In particular,
the six base change morphisms in (\ref{ar1}) are isomorphisms, so that  ${sp^*_{P^f}}$ is defined. The assumption
$\phi f_* F=0$ implies that the cones of the four horizontal arrows in rows one and two are isomorphisms, so that
all seven arrows on these two  rows are isomorphisms.  The same argument shows that the seven arrows in rows
two and three are isomorphisms. In particular, the morphisms of type $\delta$ and  
the morphisms ${^1\!\s^*_{\leq \bullet}}$ are isomorphisms, and the conclusion follows.

(3). 
We prove the first assertion in (3). The base change morphisms for $f$ are isomorphisms. The vanishing assumptions
imply that we can identify the first two rows, so that ${sp^*_{P^f}}$ is defined. We prove the second assertion
in (3).
The hypothesis $\phi f_* F=0$ implies that the morphisms of type $\delta$ are isomorphisms, and we are done.

(4). The vanishing assumptions imply that all horizontal arrows in (\ref{ar1}) are isomorphisms. It follows that all arrows
in (\ref{ar1}) are isomorphisms and the conclusion follows.

(5). Since $f$ is proper and $F$ is semisimple, $f_*F$ is semisimple by the Decomposition Theorem \ref{dt},
so that    $\ptd{\bullet} f_* F$ is a direct summand of $f_* F$. It follows that:
$\phi v_* f_* F=0$ is equivalent to  $\phi v_* \ptd{\bullet} f_* F=0$;
$v_* \phi f_* F=0$ is equivalent to  $v_*\phi \ptd{\bullet} f_* F=0$. This shows that  (5) follows from (3).

(6). Since $v$ and $f$ are proper, ${sp^*_{P^f}}$ is defined. The semisimplicity assumptions 
on $F$ and on $i^*F[-1]$ and the projectivity of $f$ imply, via Proposition \ref{crux}, that the morphisms of type $\delta$ are isomorphisms. It remains to observe that
the assumption $\phi v_* f_*F=0$, coupled with the semisimplicity of $f_*F$, implies, as seen in the proof of  (5),
that $\phi v_* \ptd{\bullet} f_*F=0$. It follows that $\s^*$  is an isomorphisms and we are done.

(7). Since $f$ is proper, the corresponding base change morphisms are isomorphisms. As seen in the proof of (6), 
the morphisms of type $\delta$ are isomorphisms. As seen in the proof of (5), the vanishing assumptions
imply that  $\phi v_* \ptd{\bullet} f_* F=0$ and  $v_*\phi \ptd{\bullet} f_* F=0$.
As in the proof of (3), the base change morphisms for $v$ are isomorphisms and, as in the proof of (1),  
the morphisms ${^1\!\s^*_{\leq \bullet}}$ are isomorphisms.
The proof of (7) is complete.
\end{proof}

\begin{rmk}\la{kaw5}
{\rm
 ({\bf Vanishing of $\phi$ and base change})
 }
The analogue of  Remark \ref{kaw3} holds in the context of the proof of Theorem  \ref{crit}.
\end{rmk}

\begin{rmk}\la{ho}$\;$
Recall that a condition of type $\phi v_* (-) =0$ means that $v_*(-) $ has locally constant cohomology
sheaves $R^\bullet v_*(-)$  on  $S$ near $s$.
\end{rmk}

\begin{rmk}\la{hoi}$\;$
\ben
\item
 \ci[Theorem 3.2.1 parts (i,ii,ii)]{dema} are  implied by the slightly more precise  Theorem \ref{crit} parts (1,3,5), respectively.
 
 \item
 Proposition \ref{tbo} parts (1,2,3) are also the special cases of Theorem \ref{crit} parts (1,3,5) when one takes $f:X\to Y$
 to be the identity.
 \een
\end{rmk}

\subsection{The specialization morphism via a compactification}\la{spcor}$\;$

Consider Theorem \ref{crit}:  parts (3,4,5,7)  do not assume that $v=v_X:X\to S$ is proper, but make some  local constancy
assumptions for certain direct images to $S$, namely $\phi v_* (-) =0$.

In this subsection we provide, in the context of  a proper $\SP$-morphism $f:X\to Y$ and of a non proper morphisms $v_X:X \to \SP$,  sufficient conditions ensuring that we obtain  well-defined $sp^*_{P}(G)$ or $sp^*_{P^f}(F)$, but that 
do not require local constancy assumptions. On the other hand, they require the existence of  good compactification.  As we shall see, here ``good" is made precise by  conditions relating  vanishing cycles at the boundary.
In general, it may be difficult to achieve such a vanishing. See Remark \ref{sn} for one situation in which this is possible. They are also achieved in the compactification of Dolbeault moduli spaces
constructed in \ci{decomp2018}; see \S\ref{apell}.

We  start with a proper $S$-morphism  $f:X^o\to Y^o$, and we  do not assume that $v=v_{Y^o}: Y^o\to S$ is proper.

In  order to circumvent   base change issues, it  is natural to first compactify the picture. We start with $f:X^o\to Y^o$ and we get, by Nagata's completion theorems:  
Zariski open and dense embeddings  $X^o \subseteq X$ and $Y^o \subseteq Y$; a proper morphism $f:X \to Y$ 
extending $f:X^o \to Y^o$; a proper morphism  $v_Y: Y \to S$  extending $v_{Y^o}: Y^o \to S.$ 
By blowing-up $Y,$ if necessary,
we may assume that $W:=Y\setminus Y^o$ supports an effective Cartier divisor. 

It follows that we can place ourselves in the following:

\begin{setup}\la{setzxw}
Let $S$ be a variety.
Consider a Cartesian diagram of $S$-morphisms, with $(a:W \to Y \leftarrow Y^o:b)$ closed/open complementary immersions:
\beq\la{hh11}
\xymatrix{
Z    \ar[r]^a    \ar[d]^f    & X   \ar[d]^f  & X^o \ar[l]_b \ar[d]^f
\\
W \ar[r]^a &  Y & Y^o \ar[l]_b
}
\eeq
such that:  all morphisms $f$ are  proper; the varieties $Y$ and $W$ proper over $S;$
the boundary $W$ on $Y$ supports
an effective  Cartier divisor. In particular,   the functors  $b_!, b_*:Y^o \to Y$ are $\frak{t}$-exact. 
\end{setup}

Many of us are used to denote the pair of closed/open embeddings $(a,b)$ by $(i,j)$. However,  in this paper, and in large part of the vanishing cycle literature, the closed embedding of the 
special fiber is systematically denoted by $i$. 

We denote by ${\rm ad}$ the distinguished triangle of endofunctors of  $D^b_c(Y)$:
\beq\la{adj}
\xymatrix{
{\rm ad}:= & a_* a^! \ar[r]^-{\frak a} & {\rm Id} \ar[r]^-{\frak b} & b_*b^* \ar@{~>}[r] &;
}
\eeq
by plugging  $G\in D^b_c(Y)$ in (\ref{adj}), we obtain the distinguished triangle ${\rm ad} (G)$ in $D^b_c(Y),$  functorial in $G$.

Proposition  \ref{r54} below  is an ``$f$-proper, but $v_X$-non-proper" analogue to Theorem \ref{crit}.(1), where it was assumed that $f$ and $v$ are proper and that $\phi F=0$.
Recall Remark \ref{f03}, which allows us to use a general point $t$ to express the target of specialization morphisms.

Note that the assumptions  (B) in Proposition \ref{r54} imply the assumptions of Theorem \ref{crit}.(5)
which lead  to $sp^*_{P^f}$ being an isomorphism. We have included (B)  for completeness in the context of this 
subsection. Recall the notational Remark \ref{short}.

\begin{pr}\la{r54}
Assume we are in the Set-up {\rm \ref{setzxw}}. Assume that  $S$ is a nonsingular and connected  curve and let $s \in S$ be a point.
Let $F\in D^b_c(X)$.  

Consider the following two sets of conditions:
\ben
\item[(A)]
$\phi b_* f_*b^* F=0$  ($\Leftarrow (\phi b_* b^* F=0)$),  and  $\phi f_* b^* F=0$ ($\Leftarrow (\phi b^*F=0)
\Leftarrow (\phi b^* F=0) \Leftarrow (\phi F=0)$.

\item[(B)] 
$F$ semisimple, $\phi v_* b_* f_* b^* F=0$ ($\Leftarrow (\phi f_* b_*b^* F=0) \Leftarrow (\phi b_*b^* F=0)$), 
and $\phi f_* b^*F=0$  ($\Leftarrow (\phi b^* F=0) \Leftarrow (\phi F=0)$).
\een
If either conditions  (A) or (B) are met, then   the $P^f$-filtered specialization morphism:
\beq\la{g42.8}
\xymatrix{
sp^*_{P^f} (F): (R\Gamma (X^o_s,    i^*  b^* F), P^{f_s})  \ar[r]  & 
(\psi v_* f_* i^* F, P) \simeq 
 (R\Gamma (X^o_t,   t^* b^* F), P^{f_t}),
}
\eeq
where $t \in S$ is general for $F$ with respect to $X/S$ and for $f_*F$ with respect to $Y/S$,
is defined and it is a filtered isomorphism for the perverse Leray filtrations.

If $f:X\to Y$ is the identity, then the same conclusion holds with  $sp^*_{P^f} (F) = sp^*_P(F).$
\end{pr}
\begin{proof} Since the proof is analogous to the proof of  Theorem \ref{crit}.(1), we only indicate the main line of the arguments.

In complete analogy with the formation  of the commutative  diagram (\ref{ar1}), we form the commutative diagram: 
(with many decorations omitted)
\beq\la{ar2}
\xymatrix{
i^*[-1] v_* b_* \ptd{\bullet}   f_* b^* \ar[r]^-{\s^*}  \ar[d]^-{bc} & 
\psi [-1] v_* b_* \ptd{\bullet}   f_* b^*  \ar[d]^-{bc}   \ar[r]^-{\s^!}  & 
i^![1]  v_* b_* \ptd{\bullet}   f_* b^* \ar[d]_-=^-{bc} & (\phi v_* b_* \ptd{\bullet} f_*b^* F)
\\
 v_*   i^*[-1] b_*   \ptd{\bullet} f_* b^* \ar[r]^-{\s^*}   \ar[d]^-{bc^{i^*b_*}}   &  
 v_* \psi[-1]  b_* \ptd{\bullet} f_* b^*   \ar[r]^-{\s^!}  \ar[d]^{bc} &  
v_* i^![1]  b_* \ptd{\bullet} f_* b^*
 \ar[d]^-{bc}_-=  \ar[u] & (v_* \phi  b_* \ptd{\bullet} f_* b^* F)
\\
v_*  b_* i^*[-1]   \ptd{\bullet}  f_* b^* \ar[r]^-{\s^*}    &  
 v_*  b_* \psi[-1]   \ptd{\bullet}  f_* b^*    \ar[r]^-{\s^!} 
\ar[d]_-=   &  
v_*  b_* i^![1]   \ptd{\bullet}  f_* b^* \ar[d]^-{\delta^!_{\leq \bullet}} \ar[u] & (\phi  f_* b^* F)
\\
  v_*  b_* \ptd{\bullet}  i^*[-1]    f_* b^*  \ar[r]^-{\s^*}     \ar[u]_-{\delta^*_{\leq \bullet}} \ar[d]^-{bc}&  
   v_*  b_* \ptd{\bullet}  \psi[-1]    f_* b^*   \ar[r]^-{\s^!}  \ar[d]^-{bc} \ar[u] &  
  v_*  b_* \ptd{\bullet}  i^![1]    f_* b^*   \ar[d]^-{bc}_-= & (\phi f_* b^* F),
\\
 v_* b_*  \ptd{\bullet} f_*    i^*[-1] b^*  \ar[r]^-{\s^*}  \ar@/0pt/@{-->}[uuuur]    & 
v_* b_*  \ptd{\bullet} f_*    \psi[-1] b^*   \ar[r]^-{\s^!}   &  
v_* b_*  \ptd{\bullet} f_*    i^![1] b^*  \ar[u]  & (\phi b^*F),
}
\eeq
the properties of which  are similar to the ones of (\ref{ar1}).
The second row from the top in (\ref{ar2}) does not have a counterpart in (\ref{ar1}), and this is because
here we need to consider the  base change morphisms  for $i^*b_*.$
The terms in parentheses for the first three rows are, up to shift,  the cones of the arrows of type $\s$. 
After having plugged an $F\in D^b_c(X)$ in (\ref{ar2}), the vanishing of the term in parentheses implies that the corresponding morphisms of type $\s$ are isomorphisms.

Since $v:X\to S$ is proper, the first two rows get identified by proper base change. Since $f$ is proper, the same is true
for the last two.

Since every row contains a vertical arrow which is an isomorphism, we only need to show that all horizontal arrows  in (\ref{ar2}) are isomorphisms. It is enough to show that
the corresponding terms in parentheses vanish. By the identifications above, we need to do so only for rows two, three and four.

The assumption $\phi b^* F=0$,  which is common to (A) and (B), together with the properness of $f$, implies the vanishing
of the terms in parentheses for the bottom three rows (for the third row, one uses that $\phi$ commutes with perverse truncation). 

We are left with proving that the morphisms of type $\s$ on row two are isomorphisms.

We prove that  (A) implies the desired conclusion.
The assumption $\phi b_* b^* F=0$ implies the vanishing
of the terms in parentheses for the second row because of what follows:
\[
0= f_* \phi b_* b^* F =  \phi  f_* b_* b^* F= \phi   b_* f_* b^* F =  \phi   b_* b^* f_* F, 
\]
so that:
\[
0= \ptd{\bullet}  \phi   b_* b^* f_* F = \phi  \ptd{\bullet}  b_* b^* f_* F =  \phi    b_* \ptd{\bullet} b^* f_* F  =
\phi    b_* \ptd{\bullet}  f_*   b^*F,
\]
where in the identities above we have used proper base change for $f$, the fact that $fb=bf$ in  the r.h.s. of (\ref{hh11}),
the $\frak t$-exactness of $b_*$  ($b$ is  affine and quasi-finite) and of $\phi$.

(B) has already been proved in Theorem \ref{crit}.(5).
\end{proof}

\begin{rmk}\la{rv03}
The proof of Proposition \ref{r54} shows that under its hypotheses,  (A) or (B), the base change morphism
$bc^{i^*b_*}: v_* i^* b_* \ptd{\bullet} f_*b^*F \to  v_* b_* i^* \ptd{\bullet} f_*b^*F$ in (\ref{oras})  is an isomorphism.
\end{rmk}

\begin{rmk}\la{boxb}
Remark \ref{box}, on the part of (\ref{ar1}) that is needed to define $sp^*_{P^f}(F)$, 
 holds essentially verbatim in the context of diagram (\ref{ar2}).
\end{rmk}

\begin{rmk}\la{sn}
The conditions  $\phi b_* b^* F=0=\phi b^*F$  in Proposition \ref{r54} are  met if, for example, $X/S$ is smooth, $X\setminus X^o$ is a simple normal crossing
divisor over $S$ and $F$  has locally constant cohomology sheaves. See \ci[XIII, Lemme 2.1.11 (dualized)]{sga72}.
We do not know of a  weaker, but similar set of conditions that leads to $sp^*_{P^f}$ being defined, but not necessarily an isomorphisms.
\end{rmk}

\begin{cor}\la{r55}
Let things be as in Proposition \ref{r54}. Then we have the following commutative diagram, with horizontal arrows given by restriction, and with vertical arrows given by the well-defined  $P^f$-filtered specialization morphisms:
\beq\la{g42.8bix}
\xymatrix{
(R\Gamma (X_t,   t^*  F), P^{f_t}) \ar[r]   &   (R\Gamma (X_t^o,   t^*b^* F), P^{f_t}) \\
(R\Gamma (X_s,    i^*   F), P^{f_s})  \ar[r] \ar[u]^-{sp^*_{P^f}}  &  (R\Gamma (X_s^o,   i^* b^* F), P^{f_s})
 \ar[u]^-{sp^*_{P^f}}_-\cong.
}
\eeq
\end{cor}
\begin{proof}
We artificially add  to digram (\ref{ar1})  one row identical to the second row from the top and place it between
the second and third row; there are now five rows; we connect the second row to the new third row by the identity;
we connect the new third row to the new fourth row the same way as rows two and three in (\ref{ar1}) are connected.
The reader can verify that:  if we apply the adjunction $Id \to b_* b^*$ to this new diagram,  call it (\ref{ar1}$'$), then we obtain what can be called
the adjunction morphism between (\ref{ar1}$'$) and (\ref{ar2}).  The conclusion follows.
\end{proof}

\subsection{Specialization morphisms and the long exact sequence of a triple}\la{tent}$\;$

In the context of the compactification Set-up \ref{setzxw}, with $S$ a nonsingular curve and $s \in S$ a point,  it is natural to ask 
about the relation between specialization morphisms and the long exact sequence in cohomology associated with the
triple $(Z, X, X^o).$ Because of base change issues, and because of the possible  failure of the morphisms
of type $\delta$ for the closed embedding $a$ to be isomorphisms, the relation is not always the expected one, i.e.
the filtered cones of the horizontal morphisms in (\ref{g42.8bix}) are not necessarily
the corresponding objects for $i^* a^! F$ and $t^* a^! F$.

Recall Remarks \ref{box}, \ref{boxfil}, \ref{boxfilp} and \ref{boxb} on the parts of the  diagrams involved in the definitions of the specialization morphisms $sp^*, sp^*_P, sp^*_{P^f}.$ 

We discuss the case of $sp^*_{P^f}$. The case $sp^*_P$ is then the special case when $f: X\to Y$ is the identity.
The case $sp^*$ is the special case, when we remove $\ptd{\bullet}$ from the picture.

\begin{lm}\la{aristo}
Let things be as in Set-up \ref{setzxw}, except that we do not assume that $v$ and $f$ are proper, nor that $W$ is Cartier.
Assume that $S$ is a nonsingular curve and $s \in S$ is a point.
The following diagram is  commutative:  
\beq\la{oras}
\xymatrix{
\psi[-1] v_* a_!a^! \ptd{\bullet} f_* 
\ar[r] 
&
\psi[-1] v_* \phantom{b_*a_*} \ptd{\bullet} f_*
\ar[r]
&
\psi[-1] v_* b_*b^* \ptd{\bullet} f_*
\ar@{~>}[r]
&
\\
i^*[-1] v_* a_!a^! \ptd{\bullet} f_* 
\ar[r] \ar[u]_-{\s^*} \ar[d]_-{bc^{i^*v_*}}
&
i^*[-1] v_* \phantom{b_*a_*} \ptd{\bullet} f_*
\ar[r] \ar[u]_-{\s^*} \ar[d]_-{bc^{i^*v_*}}
&
i^*[-1] v_* b_*b^* \ptd{\bullet} f_*
\ar@{~>}[r] \ar[u]_-{\s^*} \ar[d]_-{bc^{i^*v_*}}
&
\\
v_* i^*[-1] a_!a^! \ptd{\bullet} f_* 
\ar[r]    \ar[d]^-{i^*a_!a^! \to  a_!a^!i^*} 
&
v_* i^*[-1] \phantom{b_*a_*} \ptd{\bullet} f_*
\ar[r]    \ar[d]
&
v_* i^*[-1] b_*b^* \ptd{\bullet} f_*
\ar@{~>}[r]    \ar[d]\ar[d]^-{i^*b_*b^* \to  b_*b^*i^*}_-{bc^{i^*b_*}}
&
\\
v_*  a_!a^! i^*[-1] \ptd{\bullet} f_* 
\ar[r]                  
&
v_* i^*[-1] \phantom{b_*a_*} \ptd{\bullet} f_*
\ar[r]    \ar[u]^-=                        
&
v_*  b_*b^*  i^*[-1] \ptd{\bullet} f_*
\ar@{~>}[r] 
&
\\
v_*  a_!a^!  \ptd{\bullet} i^*[-1] f_* 
\ar[r]  \ar[u]^-{\delta^{i^*}} \ar[d]^-{bc^{i^*f_*}}             
&
v_* \ptd{\bullet} \phantom{b_*a_*}  i^*[-1] f_*
\ar[r]   \ar[u]^-{\delta^{i^*}}   \ar[d]^-{bc^{i^*f_*}}                  
&
v_*  b_*b^* \ptd{\bullet} i^*[-1]   f_*
\ar@{~>}[r]  \ar[u]^-{\delta^{i^*}}   \ar[d]^-{bc^{i^*f_*}}
&
\\
v_*  a_!a^!  \ptd{\bullet}  f_* i^*[-1]    
\ar[r]                  
&
v_* \ptd{\bullet} \phantom{b_*a_*}  f_* i^*[-1]
\ar[r]                     
&
v_*  b_*b^* \ptd{\bullet} f_* i^*[-1]
\ar@{~>}[r].  
&
\\
}
\eeq
\end{lm}
\begin{proof}
This follows at once from the following list of formal properties.

Let things be as in Set-up \ref{setzxw}.  Recall (\ref{adj}). Let $\frak i:T\to S$ be a morphism of varieties. The verification of the following
facts is formal and is left to the reader.

\ben
\item $\s^*$ is a morphism of functors.
\item
The attaching triangles  are compatible with base change. More precisely, 
the base change morphisms $\frak i^* v_* \to v_* \frak i^*$  applied to the distinguished triangle
of functors (\ref{adj})  (recall that it contains the morphisms $\frak a, \frak b$) yield morphisms of the corresponding distinguished triangles.

\item
Let $\mu: \frak i^* a_! a^! \to a_! a^! \frak i^*$ be the composition of the base change isomorphism  $\frak i^* a_! a^! 
\to  a_!\frak i^* a^!$ with the natural morphism $a_! \frak i^* a^! \to a_! a^! \frak i^*$ \ci[Proposition 3.1.9.iii]{kash}.
Then $\frak i^* \frak a= \frak a\frak i^*\circ \mu.$

\item[($3'$)]
Let $\mu': \frak i^* b_* b^* \to b_* b^* \frak i^*$ be the composition of the base change isomorphism  $\frak i^* b_* b^* 
\to  b_*\frak i^* b^*$ with the natural isomorphism $b_* \frak i^* b^* \to b_* b^* \frak i^*.$ 
Then $\mu' \circ \frak i^* \frak b= \frak b \frak i^*.$

\item
The attaching triangles are compatible with the morphisms of type $\delta.$ 
More precisely, 
the morphisms of type $\delta$ applied to the distinguished triangle
of functors (\ref{adj}) yield morphisms of the corresponding distinguished triangles.

\item
Part (2) holds   for the base change morphisms $\frak i^* f_* \to f_* \frak i^*.$
\een
\end{proof}

We need the following  remark and lemma in the proof of Proposition \ref{sonoloro}.

\begin{rmk}\la{eelui}
The base change morphism   $bc^{i^*b_*}$ in (\ref{oras})  and (\ref{ar2}) coincide.  
\end{rmk}

\begin{lm}\la{gt500}
Let things be as in the compactification Set-up \ref{setzxw} and assume that $S$ is a nonsingular connected curve and $s\in S$ is a point.
Let $G \in D^b_c(Y)$ be such that $\phi G=0$ and $\phi a^! G=0.$ Then the natural morphisms
$i^* a^! \ptd{\bullet} G \to a^! i^* \ptd{\bullet} G$ are isomorphisms.
\end{lm}
\begin{proof}
The assumptions on vanishing are equivalent to assuming that $\phi \ptd{\bullet} G=0$ and $\phi \ptd{\bullet} a^! G=0,$
so that it is enough to prove the conclusion without truncations. We have $i^*[-1] G \cong i^![1] G$ and $i^*[-1] a^! G
\cong i^![1] a^*G.$  The conclusion follows from the natural identification
$i^!a^!=a^!i^!$ as follows:
$i^*[-1] a^!G = i^![1]a^!G = a^! i^![1] G= a^! i^*[-1] G.$
\end{proof}

Recall Remark \ref{f03}, which allows us to use a general point $t$ to express the target of specialization morphisms. Recall the notational Remark \ref{short}, and convention (\ref{conv}) on shifts of filtrations.

\begin{tm}\la{sonoloro} {\rm {\bf (Specialization and long exact sequence of a triple)}}
Let things be as in the compactification  Set-up \ref{setzxw} and assume that $S$ is a nonsingular connected curve and $s\in S$ is a point.
Let $F \in D^b_c(X)$.  

Consider the following three sets of conditions:

\ben
\item[(I)]

\ben
\item
$f_*F$ has no constituents supported on
$W.$

\item
$f_* i^* F$  has no constituents supported on $W_s.$  

\item
$\phi f_* F=0$ ($\Leftarrow (\phi F=0)$).

\item
$\phi f_* a^! F=0$ ($\Leftarrow (\phi a^! F=0)$).
 

\een

\item[(II)]
\ben

\item
$f$ is projective.

\item
$F, a^!F[1]$,  $i^*F[-1]$  $a^!i^*F$ are perverse semisimple.

\item
$\phi f_* F=0$ ($\Leftarrow (\phi F=0)$).

\item
$\phi f_* a^! F=0$ ($\Leftarrow (\phi a^! F=0)$).

\een

\item[(III)]

\ben
\item
$f$ is projective.

\item
$F, a^!F[1], i^*F[-1]$ and $a^!i^* F$ are perverse semisimple.

\item
$\phi v_* b_* f_* b^*F=0$ (cf. Remark \ref{ho})
($\Leftarrow (\phi f_* b_* b^*F=0) \Leftarrow (\phi b_*b^* F=0)$).

\item
$\phi f_* b^* F=0$.
\item
$\phi f_* a^! F=0$.

\een
\een
Assume that either (I), (II), or (III) holds.
Then, for $t\in S$ general,  we have 
the isomorphism of distinguished triangles in $DF(pt):$ 
\beq\la{mttc}
\xymatrix{
\left( R\Gamma (Z_t,  t^*  a^! F), P^{f_t}(-1)\right)  \ar[r]  & 
\left( R\Gamma (X_t,  t^* F), P^{f_t}\right)  \ar[r] &
\left( R\Gamma (X^o_t,  t^* b^* F), P^{f_t}\right) \ar@{~>}[r] &
\\
\left( R\Gamma (Z_s,  i^* a^!  F), P^{f_s}(-1)\right)  \ar[r] \ar[u]^-{sp^*_{P^f}(a^!F)}_-\cong  & 
\left( R\Gamma (X_s,  i^* F), P^{f_s}\right)  \ar[r]  \ar[u]^-{sp^*_{P^f}(F)}_-\cong &
\left( R\Gamma (X^o_s,  i^* b^* F), P^{f_s}\right) \ar@{~>}[r] \ar[u]^-{sp^*_{P^f}(b^*F)}_-\cong  &.
}
\eeq
\end{tm}
\begin{proof} 
We  plug $F\in D^b_c(X)$ in diagram (\ref{oras}).
 We still denote the resulting diagram by (\ref{oras}) in what follows.  Then (\ref{oras}) consists of three columns, $C_1, C_2$ and $C_3,$ and six rows $R_1, \ldots , R_6.$
Each row is a distinguished triangle. The vertical arrows yield morphisms of distinguished triangles.

The goal is to prove that, under the assumptions of the theorem: all the vertical arrows   in (\ref{oras}) are isomorphisms.
In fact, then, by considering the compositum of the arrows from $R_6$ to $R_1$,  and in view of (\ref{r405}), we obtain the following system of isomorphisms
of distinguished triangles:
\beq\la{iola}
\xymatrix{
v_* a_! \ptd{\bullet +1} f_*  t^*[-1] a^! F \ar[r]   &
v_*  \ptd{\bullet} f_*  t^*[-1] F  \ar[r]  &
v_* b_* \ptd{\bullet} f_*  t^*[-1] b^* F  \ar@{~>}[r] &
\\
\psi [-1] v_* a_! \ptd{\bullet +1} f_*   a^! F \ar[r]  \ar[u]_-\simeq^{(\ref{r405})}&
\psi [-1] v_*  \ptd{\bullet} f_*  F  \ar[r]  \ar[u]_-\simeq^{(\ref{r405})}&
\psi [-1] v_* b_* \ptd{\bullet} f_*  b^* F \ar@{~>}[r]   \ar[u]_-\simeq^{(\ref{r405})} &
\\
v_* a_! \ptd{\bullet +1} f_*  i^*[-1] a^! F \ar[r]  \ar[u]_-\simeq^-{sp^* (\ptd{\bullet +1} a^!F)} &
v_*  \ptd{\bullet} f_*  i^*[-1] F  \ar[r] \ar[u]_-\simeq^-{sp^* (\ptd{\bullet} F)} &
v_* b_* \ptd{\bullet} f_*  i^*[-1] b^* F  \ar@{~>}[r]  \ar[u]_-\simeq^-{sp^* (\ptd{\bullet} b^*F)} &,
}
\eeq
which yields the desired conclusion (\ref{mttc}).

We first prove that (I) implies that  desired goal that all the vertical arrows are isomorphisms.

The plan of the proof  is as follows.
 Prove that  columns $C_1, C_2$ and $C_3$ coincide  with diagrams (\ref{boxs1}) 
involved in the definitions  of $sp_{P^f}^*(a^!F), sp_{P^f}^*(F)$ and $sp_{P^f}^*(b^*F)$, respectively.
Prove that all arrows in said columns are isomorphisms, so that all three morphisms $sp^*_{P^f}$
are defined and isomorphisms, and the composition of the arrows from $R_6$ to $R_1$ yields (\ref{mttc}).

{\bf Step 1:} we carry out the plan for $C_1$. In this case,  $P^f(-1)$ will enter the picture naturally.

Let us   identify $C_1$  with  (\ref{boxs1})  for $a^!F$.
In order to accomplish this, we need to show that:

\ben
\item[(i)]
  the top of  $C_1$, i.e.  $\psi [-1] v_* a_! a^! \ptd{\bullet} f_*F,$  coincides with $\psi [-1] v_* a_! \ptd{\bullet +1} f_* a^!F$;
\item[(ii)]  the morphism $v_*i^*[-1] a_! a^! \ptd{\bullet} f_* F \to  v_* a_! a^! i^*[-1] \ptd{\bullet} f_* F$ is an isomorphism;
\item[(iii)]  the bottom of $C_1$, i.e. $v_*a_! a^! \ptd{\bullet} f_* i^* [-1] F,$ coincides with $v_! a_! \ptd{\bullet +1} f_* i^*[-1] a^!F.$ 
\een

To prove (i) it is enough to prove that $a^! \ptd{\bullet} f_*F$ coincides with $ \ptd{\bullet} f_*a^!F.$
This follows from the fact that the natural morphism (\ref{j11a}) of type $\delta^{a^!}$  for the closed embedding $a$ is an isomorphism
in view of our assumptions that $f_*F$ has no constituents supported on $W,$ so that Proposition \ref{iotanoc} applies.

For (ii), we argue as follows. Since $i^*a_!=a_!i^*,$ is is enough to prove that $i^*a^! \ptd{\bullet} f_*F$ coincides with
$a^!i^* \ptd{\bullet} f_*F.$ This follows from our vanishing hypotheses and from  Lemma \ref{gt500}, in view of
the assumptions  $\phi f_*F=0$ and $\phi f_*a^!F=0$.

For (iii), we argue as we have done for (i), by using the assumption that $f_*i^* F$ has no supports on $W_s.$

We have shown that $C_1$ agrees with (\ref{boxs1})  for $a^!F$.
Since $f$ and $v$ are proper, we have that $sp^*_{P^f} (a^!F)$ is defined. 

Let us now prove that all arrows in $C_1$ are isomorphisms.

The base change maps are isomorphisms.
The assumption that $\phi F=0$ implies $\phi f_* F=0$, so that $\delta^{i^*}$ is an isomorphism by Lemma \ref{kij}.(3).
Finally, we claim that $\s^*$ on the top of $C_1$ is an isomorphism. This follows from the fact that its cone 
$\phi v_* a_! a^! \ptd{\bullet} f_* F=0$: in fact, by using the same kind of argument employed in  (i) above, this is implied by  the assumption $\phi f_* a^!F=0$. 

This completes the proof of Step 1:
all the vertical arrows in $C_1$ are isomorphisms and give rise  to $sp^*_{P^f} (a^!F)$, which is an  isomorphism for the shifted $P^f(-1)$'s.

{\bf Step 2.}
Clearly, $C_2$ is on the nose the collection  of  morphisms  in (\ref{boxs1}) involved in the definition of $sp^*_{P^f} (F)$.
Moreover, in view of our hypotheses, Theorem \ref{crit}.(1) applies and $sp^*_{P^f}(F)$ is defined and an isomorphism.
In particular, all the arrows in $C_2$ are isomorphisms.

{\bf Step 3:} we carry out the plan for $C_3$ by  following the template of Step 1.

Since $b^*$ is \'etale, $b^*$ is  $\frak t$-exact and it commutes with $f_*$. Clearly, $b^*$ commutes with $i^*$. We 
thus  have the identities:
\beq\la{trzx}
\psi[-1] v_* b_*b^* \ptd{\bullet} f_*=\psi[-1] v_* b_* \ptd{\bullet} f_*b^*, \quad 
v_*  b_*b^* \ptd{\bullet} f_* i^*[-1] = v_*  b_* \ptd{\bullet} f_* i^*[-1] b^*.
\eeq
This implies that the top and bottom of $C_3$ give rise to the  domain and target of the putative $sp^*_{P^f} (b^*F)$
as in Step 1.

Since the morphisms of rows are morphisms of distinguished triangles and the vertical arrows
in $C_1$ and in $C_2$ are isomorphisms, then, by the Five Lemma,  so are the ones in $C_3$. 

This concludes Step 3 and we have shown that (I) implies the desired goal  that all the vertical arrows are isomorphisms.

We  prove that (II) implies the same desired goal.

To do so, it is enough to prove that (II) implies (I).
Since (c) and (d) are common to (I) and (II), we need to show that (II) implies (Ia) and (Ib).
By Proposition \ref{crux}, the assumptions $f$ projective, $F$  and $a^!F[1]$ perverse semisimple imply
(Ia). For the same reason, the assumptions $f$ projective, $i^*F[-1]$ and $a^!i^*F$ perverse semisimple imply (Ib). The desired goal is thus met.

Note that instead of assuming that $a^!F[1]$ (resp. $a^!i^* F$) is perverse semisimple, it is enough to assume that $f_* a^! F[1]$ (resp. $f_*a^!i^* F$)
satisfies the conclusion of the relative Hard Lefschetz Theorem; see Remark \ref{baco}.

We prove that   (III) implies the desired goal  that all the vertical arrows are isomorphisms.

We go back to (\ref{oras}). 
Again, the goal is to prove  that the bottom and top rows are as in the proof
that (I) implies the desired conclusion, and that all vertical arrows are isomorphisms.

The top and bottom of $C_2$ are already in the desired form. 

The top and bottom of $C_3$ are already in the desired form;  see (\ref{trzx}), which   uses only that $b^*$ is $\frak t$-exact  and that it commutes with $f_*$.

By Remarks \ref{eelui} and \ref{rv03}, the base change morphism $bc^{i^*b_*}$ in (\ref{oras}) is an isomorphism.
As in the case of (II), all the morphisms of type $\delta$, connecting $R_5$ to $R_4$ are isomorphisms.
The base change morphisms for $i^*f_*$ and $i^*v_*$ are also isomorphisms in all columns, including $C_3$.

All the arrows in $C_3$ are thus  isomorphisms as soon as $\s^*$ is. This follows from the hypothesis that 
$\phi v_*b_*f_* b^* F=0$ as follows.
Since $F$ is semisimple, so is $b^*F$.
The Decomposition Theorem implies that $\phi v_* b_* \ptd{\bullet} f_* b^* F$ is a direct summand
of $\phi v_* b_*  f_* b^* F=0$, so that $\phi v_* b_*b^* \ptd{\bullet} f_* F =\phi v_* b_* \ptd{\bullet} f_* b^* F= 0$,
where we have used the  $\frak t$-exactness of $b^*$ and proper base change. The cone of $\s^*$ in $C_3$ is zero and $\s^*$ in $C_3$ is an isomorphism.

We have proved that  $C_3$ has the desired form and that all of the arrows in $C_3$ are isomorphisms, so that
$sp_{P^f}(b^*G)$ is an isomorphism.
 
Since the base change arrow $bc^{i^*v_*}$ is an isomorphism in $C_3$ it follows that all arrows connecting $R_3$ to $R_4$
are isomorphisms. In particular, the arrow denoted by  $i^* a_! a^! \to a_!a^! i^*$ is an isomorphism.

The arrow $\s^*$ in $C_1$ is an isomorphism if and only if $\phi v_* a_! a^! \ptd{\bullet} f_*F=0$, which we now prove.
As in the case of (II), we have that $f_*F$ has no constituents supported on $W$, so that $\phi v_* a_! a^! \ptd{\bullet} f_*F=
\phi v_* a_!  \ptd{\bullet +1} f_* a^!F=  v_* a_!  \ptd{\bullet +1} f_* \phi a^!F=0$, by the assumption $0=\phi f_*a^!F=
f_*\phi a^! F$ ($f$ is proper).

What above also implies that the top of $C_1$ is $\psi [-1] v_* a_! \ptd{\bullet +1} f_* a^!F$, i.e. the system
giving rise to the  target of the specialization map. As to the bottom of $C_1$, since all arrows are isomorphisms, it is identified with 
$v_* i^* [-1] a_! a^! \ptd{\bullet} f_* F$, which, by what above, is identified with 
$v_* a_! i^* [-1] \ptd{\bullet +1} f_* a^! F$, which --using the fact that $f_* a^! F$ has no constituents supported on 
$W_s$ by virtue of the hypotheses that $f$ is projective,  $a^!F[1]$ and $i^* a^! F$ are perverse  semisimple,
coupled with Proposition \ref{crux}-- equals  $v_* a_!  \ptd{\bullet +1} f_* i^* [-1] a^! F$, i.e. the system giving rise to the  source of $sp^*_{P^f} (a^!F)$. 

We have proved that $C_1$ has the desired form and that  all arrows in $C_1$ are isomorphisms.
We have already proved that the same holds for $C_3$.  As seen earlier by using the Five Lemma, it follows that the same holds for $C_2$.
The conclusion follows.
\end{proof}

\begin{rmk}\la{bbb}
Condition (I) in Proposition \ref{sonoloro} is met if: $f$ is projective,  $F$ is perverse semisimple on $X$ and  $a^*[-1]F$
is perverse semisimple on $Z$. Condition (II)  in Proposition \ref{sonoloro} is met if: $f$ is projective,  $i^*F[-1]$ is perverse semisimple on $X_s$ and  $a^*i^*[-2]F$
is perverse semisimple on $Z_s$. Both cases follow from Proposition \ref{crux}.
\end{rmk}

\begin{rmk}\la{bbbo}
Proposition \ref{sonoloro}  yields a different proof, under different hypotheses, of  the conclusions of both Proposition \ref{r54} and Corollary
\ref{r55}.
\end{rmk}

\begin{rmk}\la{ert}
{\rm The toy-model for Theorem \ref{sonoloro} is   when we assume
that  $v_X, v_Y$, $v_Z$, $v_W$ and $f$ are proper and smooth.   We do not assume that $v_{Y^o}$ is proper.
In this case, we leave to the reader to verify that  if we set $F=\rat_X$, and we consider the usual Leray filtration on the cohomology
of $X_s$ and $X_t$ ($Z_s, Z_t$, resp.), with respect to $f_s:X_s \to Y_s$ and $f_t:X_t \to Y_t,$ ($f_s:Z_s \to W_s$ and $f_t:Z_t \to W_t$, resp.), then
we end up with: the specialization morphisms are defined, they are filtered isomorphisms,  and they fit in the   long exact Gysin sequences: (in what follows $\star$ is arbitrary, but fixed)
\beq\la{gys}
\xymatrix{
\ldots \ar[r] & L_{\star -2} H^{\bullet -2} (Z_t,\rat) \ar[r]^-{\rm Gysin} & L_\star H^{\bullet} (X_t,\rat) \ar[r] &
L_\star H^{\bullet}(X_t^o,\rat) \ar[r]^-{+1} & \ldots
\\
\ldots \ar[r] & L_{\star -2} H^{\bullet -2} (Z_s,\rat) \ar[r]^-{\rm Gysin} \ar[u]_-{{\rm sp}_Z} & 
L_\star H^{\bullet} (X_s,\rat) \ar[r]  \ar[u]_-{{\rm sp}_X} &
L_\star H^{\bullet} (X_s^o,\rat) \ar[r]^-{+1} \ar[u]_-{{\rm sp}_{X^o}}  & \ldots
}
\eeq
where $L$ stands for the increasing classical Leray filtration (which classically starts at zero).
}
\end{rmk}

\section{Applications to the Hitchin morphism}\la{apell}$\;$

In this section we apply the results of \S\ref{spsfi}, especially equation (\ref{mttc}) in  Theorem \ref{sonoloro}
to the triples  arising from the compactification of Dolbeault moduli spaces in families 
\ci{decomp2018}. The main result is Theorem \ref{lezse}, to the effect that the triple given by the Dolbeault moduli  spaces over a curve, its compactification and the boundary gives rise to  a specialization morphism of long exact sequences for the   (intersection) cohomology of the special and general points that is a filtered isomorphisms
of triples for the perverse Leray filtrations. This is what one would obtain if we were in the oversimplified and  ideal situation of a  fiber bundle with boundaries, and we were to take the Leray filtrations.

The compactification \ci{decomp2018}  is obtained by taking a $\Gm$-quotient; this is quickly reviewed in
\S\ref{codomosp}.
In order to work with such quotients, we need some results on descending objects and properties along these quotients, which may be of some independent interest; see \S\ref{desclem}. \S\ref{sppro} contains some special properties of the constant sheaf and of the intersection cohomology complex for the quotients we obtain.
Finally, we put everything together in \S\ref{icq}, where we prove Theorem \ref{lezse}.

\subsection{Descending along \texorpdfstring{$\Gm$}{gm}-quotients}\la{desclem}$\;$

In this section, we assume we have two Cartesian diagrams of morphisms over a variety $S$ as in the compactification Set-up \ref{setzxw}:
\beq\la{sicomp1}
\xymatrix{
\ms{Z}  \ar[d]^-{\hh} \ar[r]^-a &    {\ms{X}}   \ar[d]^-{\hh}  &    \ms{X}^o \ar[d]^-{\hh} \ar[l]_-b &
Z  \ar[d]^-{\hh} \ar[r]^-a &   X   \ar[d]^-{\hh}  &   X^o  \ar[d]^-{\hh} \ar[l]_-b 
\\
\ms{W}   \ar[r]^-a & {\ms{Y}}  &      \ms{Y}^o \ar[l]_-b &
W   \ar[r]^-a & Y  &      Y^o \ar[l]_-b,
}
\eeq
such that: $\Gm$-acts equivariantly on the l.h.s. with finite stabilizers;  the actions cover the trivial action over $S;$
the r.h.s. is obtained by taking the quotient under the $\Gm$-action; the quotient morphisms $\pi$ are geometric quotients.

\begin{lm}\la{zz0}
Each quotient morphism $\pi$ above   factors as $\pi = qp,$ where $p$ is a quotient $S$-morphism by a finite 
subgroup $C$ of $\mathbb G_m,$
and $q$ is a smooth quotient $S$-morphism by a $\mathbb G_m$-action, so that, for example, we have:
\beq\la{k1}
\xymatrix{
\pi: \ms{X} \ar[r]^-p & \ms{X}':=\ms{X}/C \ar[r]^-q& X=\ms{X}/\Gm =\ms{X}'/\Gm'(:=\Gm/C).
}
\eeq
\end{lm}
\begin{proof}
Let  $C\subseteq \Gm$ be any finite subgroup  containing all the stabilizers (recall that we are in a finite type situation).
The variety  $\ms{X}':= \ms{X}/C$  over $S$ is endowed with the induced $\Gm':= \Gm/C\simeq \Gm$-action, and  we have  diagram (\ref{k1}).
The quotients morphism $p$ is finite.  By Luna \'Etale Slice Theorem,  the quotient morphism $q$ is smooth:  the 
$\mathbb G'_m$-action has trivial stabilizers,  so that $q$ is 
\'etale locally a projection map with scheme-theoretic fiber $\mathbb G_m$.
\end{proof}

We record the following remark for use in the proof of  Lemma \ref{jhhje}.

\begin{rmk}\la{forlu}
The following commutative diagrams are Cartesian in the category of topological spaces:
\beq\la{mm1}
\xymatrix{
\ms{Z} \ar[r]^-a \ar[d]_-p 
& \ms{X} \ar[d]_-p
& \ms{X}^o \ar[l]_b \ar[d]_-p
\\
\ms{Z'} \ar[r]^-a \ar[d]_-q 
& \ms{X'} \ar[d]_-q
& \ms{X'}^o \ar[l]_b \ar[d]_-q
\\
Z \ar[r]^-a  
& X 
& X^o \ar[l]_b. 
}
\eeq
The same is true if we replace $(Z,X,X^o)$ etc. with $(W,Y,Y^o)$ etc.
\end{rmk}

\begin{lm}\la{gao} {\rm ({\bf Descending the lack of constituents on the boundary  to a quotient by $\mathbb G_m$})}
Let $\ms{G} \in D(\ms{Y})$ and $G \in D^b_c(Y)$  be such that $q^* G$ is a direct summand of $p_* \ms{G}.$ If
$\ms{G}$ has no constituents supported on the boundary  $\ms{W}$, then  $G$ has no constituents supported on the boundary $W$.
\end{lm}
\begin{proof}
{\em CLAIM: $p_* \ms{G}$ has no constituents supported on $\ms{W}'.$}
Since $p$ is finite, $p_*$ is $\frak t$-exact, so that $p_*$ and $\ph^{\bullet}$ commute. It follows that we may assume that $\ms{G}$ is perverse.
Let $J_\bullet \ms{G}$ be a Jordan-Holder filtration for $\ms{G},$ so that the non trivial  graded objects $Gr^J_\bullet \ms{G}$
are simple, i.e. they are intersection complexes supported at irreducible closed subvarieties of $\ms{Y}$ and with simple 
coefficients. 
By the $\frak t$-exactness of $p_*$, we have that $p_* J_{\bullet} \ms{G}$ is a filtration
and $Gr^{p_*J}_\bullet p_* \ms{G}=p_* Gr^J_\bullet \ms{G}$. Since $p$ is finite, hence small, we have that
$p_*$ sends the  intersection complex of a   closed subvariety $\ms T$ of $\ms{Y}$ with twisted coefficients, to the intersection complex of  the image $p(\ms T)$ with appropriately twisted coefficients.
By construction, we have that  $\ms{W} = p^{-1} (\ms{W}').$
The subvarieties $\ms T$ appearing as the supports of the simple $Gr^J_\bullet \ms{G}$ above are not contained in $\ms{W}$ by assumption, so that their  $p(\ms T)$'s are not contained in $\ms{W'}$.
The claim follows.

By construction, we have that  $\ms{W}' = q^{-1} (W).$
Since $q^*G$ is assumed to be a direct summand of $p_* \ms{G}$, to prove the lemma we need to show 
 that the constituents of $q^*G$ are $q^*[1]$ of the constituents of $G$. But this follows from the fact
 that, since $q$ is smooth of relative dimension one with connected fibers,  $q^*[1]$ is $\frak t$-exact and it preserves simple objects (cf. \ci[p.106, bottom]{bbd}).
\end{proof}

We need the following general result in the proof of Lemma \ref{ga}.

\begin{lm}\la{gr}
Let $C$ be a finite group acting on a variety $X$, let $p: X\to X'$ be the quotient morphism.
 Then: (i)  $p_* \rat_X$ admits $\rat_{X'}$ as a direct summand; 
(ii) $p_* IC_{X}$ admits $IC_{X'}$ as a direct summand.
\end{lm}
\begin{proof}
We prove (i). By \ci[(5.1.1), p.108]{grto}, we have that $\rat_{X'}=(p_* (\rat_X))^C$. We split the natural adjunction 
morphism $\rat_{X'} \to p_* \rat_X$ by sending a section $s \in \rat_X( p^{-1}(U))$ to  $(1/|C|)\sum_{c\in C} c^*\cdot s$.

We prove (ii). One can argue as above; we leave this to the reader. We prove a stronger statement,  which may be of independent interest, namely that the conclusion remains valid if we assume
that $p$  is a small morphism that has the property that every irreducible component 
of $X$ maps onto an irreducible component of $X'.$ Since $IC_X$ is the direct sum of the intersection complexes of the irreducible components
of $X$ and similarly for $IC_{X'},$ we may assume that $X$ and $X'$ are irreducible, and  that $p$ is small and surjective.
By the conditions of support/co-support characterizing intersection complexes, the fact that $p$ is small and surjective implies that $p_*IC_X$ is an intersection complex with some twisted coefficients $L$.
Since $p$ is small and surjective,  there is a Zariski dense open subset $j:U'\subseteq X'$ that is nonsingular,
and over which $p$ is finite and   \'etale.  Then $L$ can be taken to be the local system on $U$ associated with the topological covering, i.e. the direct image of the constant sheaf. Note that this splits off $\rat_U$   by a standard trace argument.
We now show that this splitting extends uniquely to a splitting of the corresponding intersection complexes, which proves our contention.
Let $I', J' \in P(X')$ be two intersection complexes with support precisely $X'$. We have the standard identity:
$Hom (I',J')=Hom (j^*I', j^*J')$: we have  $j_* j^* J' \in {^p\!D}^{\geq 0} (X)$; so $Hom (I', j_* j^* J')=
Hom (I', {\ph^0(j_*j^*J')})$; we have the short exact sequence $0 \to J' \to  {\ph^0(j_*j^*J')} \to Q \to 0$, where
$Q \in P(X')$ has support strictly contained in $X'$; we thus have $Hom (I',J')= Hom (I', {\ph^0(j_*j^*J')})$
because $Hom (I',Q)=0$ by reasons of support, and $Hom (I', Q[-1])=0$ by the axioms of $\frak t$-structure.
We apply the standard identity above to $I'=p_* IC_X$ and to $J'=IC_{X'}$, and we lift the splitting off of $\rat_U$
from $L$ to the splitting off of $IC_{X}$ from $p_* IC_{X}$.
\end{proof}

\begin{lm}\la{ga} {\rm ({\bf Descending $\phi =0$ to a quotient by $\mathbb G_m$})}
$\;$
Let $\ms{X}/S$ be as in (\ref{sicomp1}).
Assume $S$ is a nonsingular curve, and let  $s \in S$ be a point.  Then we have the following implications:

\ben
\item
If 
$\phi \rat_{\ms{X}} =0,$ then $\phi \rat_{X} =0$.
\item
If
$\phi {IC}_{\ms{X}} =0,$ then $\phi {IC}_{X} =0$.

\item
Let $\ms{F} \in D(\ms{X}) , F \in D^b_c(X)$  be such that $q^* F$ is a direct summand of $p_* \ms{F}$, then: if
$\phi \ms{F}=0,$ then  $\phi F=0$. 

\een
\end{lm}
\begin{proof} $\;$
We prove (3). Since $p$ is proper, we have that $\phi p_* \ms{F}=p_* \phi \ms{F}=0$.
It follows that, since $\phi$ is additive, we have that $\phi q^*F=0.$
Since $q$ is smooth, we have $q^* \phi F = \phi q^*F=0.$ Since $q$ is surjective,  we have that $q^* \phi F=0$
implies the desired conclusion (3).
Now, (1) and (2)  follow from (3) coupled with Lemma \ref{gr}, and the facts $q^*\rat_X= \rat_{\ms{X}'}$
and $q^*IC_X= IC_{\ms{X}'}[-1]$ ($q$ is smooth of relative dimension one; the shifts are innocuous).
\end{proof}

\begin{lm}\la{green}{\rm ({\bf Descending other identities})}$\;$
\ben
\item
If $a^! \rat_{\ms X}= \rat_{\ms{Z}}[-2]$, then $a^! \rat_{X}= \rat_{Z}[-2]$. 
\item
If $a^! IC_{\ms X}= IC_{\ms{Z}}[-1]$, then $a^! IC_{X}= IC_{Z}[-1]$.
\een
\end{lm}
\begin{proof}
We prove (1).


We apply $p_*$ and take invariants:
$(p_* a^! \rat_{\ms X})^C=(a^! p_* \rat_{\ms X})^C=a^! (p_* \rat_{\ms X})^C= a^! \rat_{\ms X'}$;
$(p_* \rat_{\ms{Z}})^C=\rat_{\ms{Z'}}.$ We thus have that: $a^! \rat_{\ms X'}= \rat_{\ms{Z'}}[-2]$.

Since $q$ is smooth of relative dimension $1,$  we have that $q^*[2]=q^!$,  so that $q^*$ and $a^!$ commute.
It follows that we also have that:
$q^* a^! \rat_X= a^! \rat_{\ms{X'}}= \rat_{\ms{Z'}}[-2] = q^* \rat_Z [-2]$, 
or  $q^* a^! \rat_X [2]= q^*\rat_Z$. Since $q$ is surjective, we have that
$a^!\rat_X[2]$ is a sheaf. Since $q$ is smooth and surjective, we have that 
$a^!\rat_X[2]$ is locally constant of rank one. Since the fibers of $q$ are connected, we deduce that 
$a^!\rat_X[2]$ is constant and this proves (1).

We prove (2).

By repeating the first part of the proof of part (1), we see that: $a^! IC_{\ms X'}= IC_{\ms{Z'}}[-1]$.

By repeating the second part of the proof of part (1), we see that:
$q^* a^! IC_X = q^*IC_Z[-1]$, which we re-write as $q^*[1] a^! IC_X[1]  = q^*[1]IC_Z$.

By \ci[Corollaire 4.1.12]{bbd}, $a^! IC_X[1]$ is perverse. By \ci[Proposition 4.2.5]{bbd}, $q^*[1]: P(Z)\to P(\ms{Z'})$
is fully faithful, hence conservative. This implies the desired conclusion (2).
\end{proof}


\subsection{Compactification of Dolbeault moduli spaces via \texorpdfstring{$\mathbb G_m$}{gm}-actions}\la{codomosp}$\;$

Let us  summarize the main features of the compactification in \ci{decomp2018} that we need in this section.
The notation we use here is adapted to the present needs, and  differs from the one in  \ci{decomp2018}.

Let $X^o/S$ be the Dolbeault moduli space associated with a reductive group $G$ and  a smooth projective family over $S.$
Let $f: X^o \to Y^o$ be the corresponding  Hitchin $S$-morphism; it is projective.

The main construction in \ci{decomp2018} yields two diagrams as in (\ref{sicomp1}),
so that, 
in particular, $X/S$ is a projective completion over $S$  of $X^o/S,$ with boundary the  $S$-relative Cartier divisor $Z,$  which is  proper over $S.$ Similarly for $(W,Y,Y^o)$. Let us give some more details. Let $o_S \to Y^o$ be the canonical section:
each $Y^o_s$ has a canonical distinguished point; e.g. if $G=GL_n$, then this point corresponds to the characteristic polynomial $t^n$, which in turn corresponds to nilpotent Higgs fields. We have the closed $S$-subvariety $f^{-1}(o_S) \subseteq X^o$. 
We have $\Gm$-equivariant open  inclusions and equalities as follows:
\beq\la{rata}
X^o \times (\mathbb A^1 \setminus\{0\}) = \ms{X}^o \subsetneq
\ms{X}^o \cup (X^o\setminus f^{-1} (o_S))\times \{0\} =  \ms{X}
 \subsetneq X^o \times \mathbb A^1, 
\eeq
\beq\la{rata1}
Y^o \times (\mathbb A^1 \setminus\{0\}) = \ms{Y}^o \subsetneq
\ms{Y}^o \cup (Y^o\setminus o_S)\times \{0\} =  \ms{Y}
 \subsetneq Y^o \times \mathbb A^1, 
\eeq
\beq\la{rata3}
 \ms{Z}= (X^o \setminus f^{-1} (o_S))\times \{0\} \subsetneq X^o \times \{0\}, \quad 
 \ms{W}= (Y^o \setminus o_S)\times \{0\} \subsetneq Y^o \times \{0\}.
\eeq

\begin{rmk}\la{itispo}
Note that we are not ruling out that  $\ms{Z}=\emptyset$. In this case,  our goal, i.e. Theorem \ref{lezse},
still holds and it is  not trivial, for it  states that the $P^f$-filtered specialization morphism is an isomorphism,
whether we take singular or intersection cohomology groups with rational coefficients.
If $\ms{Z}$ is not empty, it could still happen that there are points $s\in S$ such that  $\ms{Z}_s$ is not dense in every irreducible component of $X^o_s \times\{0\};$ this is not an issue in what follows.
\end{rmk}

\subsection{Special properties  of \texorpdfstring{$\rat_X$}{QX} and \texorpdfstring{$IC_X$}{icx} for the compactification}\la{sppro}$\;$

\begin{lm}\la{gbq}
\beq\la{lki9}
\xymatrix{
a^! \rat_X = \rat_Z [-2], & a^! IC_X = IC_Z [-1].
}
\eeq
\end{lm}
\begin{proof}
By virtue of
Lemma \ref{green}, it is enough to show that: (1)  $a^! \rat_{\ms X} = \rat_{\ms Z }[-2]$, and (2)  $a^! IC_{\ms X} = IC_\ms{Z} [-1]$.

In view of the open immersions $\ms{Z}= (X^o \setminus f^{-1}(o_S))\times \{0\}) \subseteq X^o \times \mathbb \{0\}$, 
and $\ms{X} \subseteq  X^o \times \mathbb A^1$, it is enough to prove (1,2) with $a: \ms{Z}\stackrel{\subseteq}\to  \ms{X}$ replaced by 
$a: X^o \times \{0\} \stackrel{\subseteq}\to  X^o \times \mathbb A^1$.  One is easily reduced to the case $a: \{0\} \stackrel{\subseteq}\to \mathbb A^1$, where the desired conclusions are trivial.
\end{proof}

The following lemma makes precise the relation between the irreducible components of the varieties in the triple
$(Z,X,X^o)$. Let  $J$ (resp. $K$, resp. $J^o$) be the set of irreducible components of $X$ (resp. $Z$, resp. $X^o$).

\begin{lm}\la{irrco}$\;$

\ben
\item
The function $J \to J^o$, $X_j \mapsto X_j \cap X^o=: X_j^o$ is well-defined and a bijection. 
We have that $X_j \cap Z=: Z_j$ is either empty, or an irreducible component of $Z$; if $j,j' \in J$ are such that $Z_j= Z_{j'}
\neq \emptyset$,  then $j=j'$; in particular, this defines an injection $K \to J$, where $Z_k = X_{j(k)} \cap Z$. 
\item
The natural adjunction disitinguished triangle (\ref{adj}) 
 $a_! a^! IC_X \to IC_X \to b_*b^* IC_X \rightsquigarrow$ reads as follows:
\beq\la{naadm}
\xymatrix{
 \bigoplus_{K}a_!  IC_{Z_k} [-1] \ar[r]  &
\bigoplus_{J} IC_{X_j}   \ar[r] &
\bigoplus_{J} b_* IC_{X^o_j} \ar@{~>}[r]&,
}
\eeq
where the arrows are direct sum arrows (i.e. the components of each arrow with pairs of distinct indices are zero).
\een
\end{lm}
\begin{proof}
Since $\ms{Z}$ is open in $X^o= X^o \times \{0\}$, we have that the set of irreducible component of $\ms{Z}$ injects
in the evident fashion
into  the one of $X^o$.  Since $\ms{X}$ is open and dense in $X^o \times \mathbb A^1$, 
we have that the set of irreducible component of $\ms{X}$ is naturally identified with the one
of $X^o \times \mathbb A^1$, which in turn, since $\mathbb A^1$ is irreducible,  is  naturally identified with the one of $X^o$.
It follows that the desired statement holds if we replace $(Z,X,X^o)$ with $(\ms{Z}, \ms{X}, \ms{X}^o)$.
The first triple is obtained from the second by dividing by the action of the  connected  group $\Gm$; the quotient morphism
has the orbits as fibers (in fact, it is a geometric quotient).
It follows that the sets of irreducible components and their mutual relations are preserved by passing to the quotient,
and (1) is proved.

We have the Cartesian diagram:
\beq\la{irrcpz}
\xymatrix{
\widehat{Z}:=\coprod_{K} Z_k \ar[d]^-\nu \ar[rr]^-a &&\widehat{X}:= \coprod_{J} X_j \ar[d]^-\nu &&
\widehat{X^o}:= \coprod_{J} X^o_j \ar[d]^-\nu  \ar[ll]_-b
\\
Z = \cup_{K} Z_k \ar[rr]^-a &&  X=\cup_{J} X_j &&
X^o=\cup_{J} X^o_j \ar[ll]_-b,
}
\eeq
where: $\nu$ are the evident morphisms induced by the closed embeddings of the irreducible components $Z_j\to Z$ and 
$X_j \to X$; $a$ (resp. $b$) are the evident closed (resp. open) embeddings.
We have $IC_X=\nu_* IC_{\hat{X}}$ and $IC_Z = \nu_*  IC_{\hat{Z}}$.
By base change,  by the additivity of adjunction morphisms, and by the identity $a^! IC_{X_k}=IC_{Z_k}[-1]$
--which is due to Lemma \ref{gbq}, which remains valid for the irreducible components, by virtue of the just-proved  part (1)--, we have
that the distinguished triangle (\ref{adj}) applied to $IC_X$ reads as  (\ref{naadm}), and (2) is proved.
\end{proof}

\begin{lm}\la{jhhje}
Let $S$ be a nonsingular curve and let $s\in S$ be a point.
Let $F \in D^b_c(X)$ be either $\rat_X,$ or  $IC_X.$ 
Then:
\beq\la{oq23}
\phi a^! F = 0, \quad \phi F =0, \quad \phi b_* b^* F=0.
\eeq
In particular, we have:
\beq\la{oq23bix}
\phi v_* a_! f_* a^! F = 0, \quad \phi v_* f_* F =0, \quad \phi v_*  b_* f_* b^* F=0.
\eeq

\end{lm}
\begin{proof} We refer to \ci{dema} and to  \ci{decomp2018} for references.

We have the evident morphisms $\ms{X} \to X^o \times \mathbb A^1 \to X^o$.
Let $F^o:= b^*F=F_{|X^o}.$ Let $\ms{F} \in D(\ms X)$ be the pull back of $F^o$ to $\ms{X};$
if  $F^o=\rat_{X^o},$ then $\ms{F}=\rat_{\ms{X}};$ if 
$F^o=IC_{X^o},$ then $\ms{F}=IC_{\ms{X}}[-1].$
By the Non Abelian Hodge Theorem, $X^o$ is topologically locally trivial over $S,$ so that
$\phi F^o=0$: this is clear when $F=\rat_X$, so that $F^o=\rat_{X^o}$; when $F=IC_X$,  so that $F^o=
IC_{X^o}$, we invoke the topological invariance
of the intersection complex  in the not necessarily irreducible context \ci{dema}.  
Since $pr_{X^o}: X^o \times \mathbb A^1 \to X^o$ is smooth, we have that $\phi pr_{X^o}^* F^o= pr^*_{X^o} \phi F^o=0$.
Since $\ms{X}$ is open in $X^o \times \mathbb A^1,$ we have that 
$\phi \ms{F}=0.$ By Lemma \ref{ga}.(1).(2), applied to $\pi=
qp:
\ms{X} \to \ms{X}' \to X$ and $\ms{F}$ and $F$, we see that $\phi F=0.$ Since $v$ and $f$ are proper, this implies that
$\phi v_* f_* F= v_* f_* \phi F=0.$

We now wish to apply Lemma \ref{ga} to $\pi= qp:\ms{Z} \to \ms{Z}' \to Z$,  $a^! \ms{F}$ and $a^!F$.
By the base change identities associated with (\ref{mm1}), we have that $p_* a^! \ms{F} = a^! p_* \ms{F}$ 
(always valid)
and that $q^*a^! F = a^!q^*F$ (because $q$ is smooth). From what above, we know that $p_* \ms{F}$ admits $q^*F$ as a direct summand, so that $p_* a^! \ms{F}$ admits $q^*a^! F$ as a direct summand.
We claim that $\phi a^! \ms{F}=0.$ Since $\ms{Z}$ is open in $X^o \times \{0\},$ it is enough to show that 
$\phi \w{a}^! \ms{F}=0$, where $\w{a}: X^o \times \{0\} \to X^o \times \mathbb A^1$.
Since $pr_{X^o}$ is smooth of relative dimension one, so that $pr_{X^o}^!=pr_{X^o}^*[-2]$, we see that 
$\w{a}^! \ms{F} = F^o[-2]$. Since we know that $\phi F^o=0,$ we see that $\phi \w{a}^! \ms{F}=0$, as desired.
By Lemma \ref{ga}, we have that $\phi a^!F=0$.  Since $v$, $a$  and $f$ are proper, this implies that
$\phi v_* a_! f_* a^!F= v_* a_! f_* \phi a^!F =0.$

The remaining statements follow from what we have proved and the distinguished triangle (\ref{adj}): we get 
$\phi b_*b^*F=0$, and then ve apply again $v_* f_*$ to $\phi b_*b^* F$ to conclude.
\end{proof}

\begin{lm}\la{f023}
Let $F \in D^b_c(X)$. 
Then:
(1) $f_*F$ has no constituents supported on $W$; (2) $f_*i^*F$ has no constituents supported on $W_s.$
\end{lm}
\begin{proof}
We have Cartesian diagrams:
\beq\la{mn1}
\xymatrix{
\ms{X} \ar[d] \ar[r] 
&X^o \times \mathbb A^1 \ar[d] \ar[r] 
& X^o \ar[d] 
&\ms{X}_s \ar[d] \ar[r] 
&X^o_s \times \mathbb A^1 \ar[d] \ar[r] 
& X^o_s \ar[d] 
\\
\ms{Y}  \ar[r] 
&Y^o \times \mathbb A^1  \ar[r] 
& Y^o
&\ms{Y}_s  \ar[r] 
&Y^o_s \times \mathbb A^1  \ar[r] 
& Y^o_s .
}
\eeq
Let $F^o:=b^*F=F_{X^o}$.  Denote by $\ms{F}$ the pul-back of $F^o$ to $\ms{X}$.
By proper base change, $f_* pr_{X^o}^* F^o= pr_{X^o}^* f_*F$ has no constituent on $Y^o\times \{0\}.$
In view of (\ref{rata3}), the same is true for $f_* \ms{F}.$ 
We apply  Lemma \ref{gao}, so that $f_*F$ has no constituents supported on $W.$ The proof of the second assertion is identycal in view of the fact that: if $s \in S$ is a point, then we have the $-_s$-version of (\ref{sicomp1}) and it enjoys  similar properties; in particular,  
Lemma \ref{gao} applies to it.
\end{proof}

\subsection{The long exact sequence of  the triple \texorpdfstring{$(Z,X,X^o)$}{(Z,X,X)}}\la{icq}$\;$  

In this section, we start with a smooth family of projective manifolds over a nonsingular connected curve $S$, so that, according to \S\ref{codomosp},
we have the situation in (\ref{sicomp1}), where $X^o/S$ is the Dolbeault moduli space 
for the family. We fix a point $s\in S$ (the special point) and we let $t\in S$ be a general point.

\ci[Theorem 1.1.2]{dema} proves that the specialization morphism $sp^*_{P^f}$  in intersection cohomology for the Dolbeault moduli spaces
of a smooth family of projective manifolds is an isomorphism:
\beq\la{eerr}
\xymatrix{
sp^*_{P^f}: \left(I\!H^\bullet (X^o_s, \rat),P^{f_s}\right)  \ar[r]^-\sim & \left(I\!H^\bullet (X^o_t, \rat), P^{f_t}\right).
}
\eeq

 This is accomplished by applying
\ci[Theorem 3.2.1, part (ii),  or part (iii)]{dema}. In fact,  part (ii) implies 
(iii);  this latter seems like a more natural statement, at least when dealing with semisimple coefficients $F \in D^b_c(X)$. 
Since the semisimplicity is essential in what above, the methods of  \ci{dema}  do not seem to afford the same kind of results for the non-semisimple $F=\rat_{X^o}$,
i.e. for singular cohomology.

As it is noted in Remark \ref{hoi}, 
Theorem \ref{crit} is an amplification of \ci[Theorem 3.2.1]{dema} and it can thus  be used to prove 
 (\ref{eerr})): in fact we can use any of the  parts (3,4,5,7)  of  Theorem \ref{crit} to deduce
(\ref{eerr}).
However,
none of these approaches yields a genuine new proof  of (\ref{eerr}) since, as it turns out, in verifying
the assumptions, we end up verifying the assumptions of part (5), and thus end up using \ci[Theorem 3.2.1.(iii)]{dema}.

In this section, we show that:

\ben

\item By using a  good compactification of  Dolbeault moduli  spaces, Proposition \ref{r54} yields a new proof of (\ref{eerr})); see Proposition \ref{r4z}. In fact, (\ref{eerr}) is made more precise in (\ref{g00}). The proof
in Proposition \ref{r4z} uses Proposition \ref{r54}.(A); the reader can verify that, in the case of (topologist's) intersection cohomology, one can use
Proposition \ref{r54}.(B) as well.

\item
Proposition \ref{r54}
also leads to the proof of  
a new fact -not seemingly affordable by the methods of \ci{dema}, nor of Theorem \ref{crit}-, namely that
(\ref{g00})  holds for ordinary singular cohomology as well. 

\item
The good compactification of Dolbeault moduli spaces gives rise to the main new  fact proved in this section that the specialization morphisms
$sp^*_{P^f}$ for the triple $(Z,X,X^o)$, in (the topologist's) intersection cohomology as well as in rational singular cohomology,
give rise to the  isomorphism (\ref{g11}) of long exact sequences of Theorem \ref{lezse}
\een

Recall that $P^f$ on $H^\bullet (X,F)$ is defined by considering the system $\ptd{\bullet} f_*F$. 

We  shall  make implicit use of  (\ref{tsh}) and (\ref{conv}).

In the remainder of this section:
when dealing with  intersection cohomology groups,   all morphisms are direct sum morphisms with respect to the canonical decomposition according to irreducible components; in Proposition \ref{r4z} and Theorem \ref{lezse}, the horizontal arrows
are well-defined without ambiguities, the vertical ones are well-defined modulo the ambiguities introduced by the action of monodromy exchanging the irreducible components. These ambiguities are harmless for the purposes of this paper; see Remark \ref{f01}. Moreover, they are removed if instead of using a general point $t$, we use the nearby cycle functor $\psi$.

We state the following proposition using (hyper)cohomology; the reader should have no difficulty
re-writing the stronger version of  (\ref{g00}) in $DF(pt)$ that  uses $(R\Gamma (-,-), P^f)$.
Even ignoring the l.h.s. of (\ref{g00}),  as pointed out above,  this result, i.e. that the perverse filtered specialization morphisms exists
and is an isomorphism for the singular cohomology of Dolbeault moduli spaces, is new.

\begin{pr}\la{r4z}
Let $\ms H(-)$ denote either the singular cohomology, the intersection cohomology, or the topologist's intersection cohomology
groups.
We have the following commutative diagram of filtered  finite dimensional rational vector spaces, where the horizontal arrows are the natural restriction morphisms, the vertical arrows are the well-defined perverse-filtered specialization morphisms and they are 
both isomorphisms:
\beq\la{g00}
\xymatrix{
\left(\ms{H}^{\bullet } (X_t), P^{f_t}\right) \ar[r]   &   \left(\ms{H}^\bullet (X^o_t) P^{f_t}\right) \\
\left(\ms{H}^\bullet (X_s), P^{f_s}\right)   \ar[r] \ar[u]^-{sp^*_{P^f}}_-\simeq  &  \left(\ms{H}^\bullet (X^o_t), P^{f_s}\right)
 \ar[u]^-{sp^*_{P^f}}_-\simeq.
}
\eeq 
In the case of (the topoligist's) intersection cohomology, the arrows are direct sum arrows for the canonical decompositions into irreducible components (\ref{naadm}).
\end{pr}
\begin{proof}
By Lemma \ref{jhhje}, coupled with Theorem \ref{crit}.(1), we have that the perverse filtered specialization morphism
on the l.h.s. of (\ref{g00}) is defined and it is an isomorphism.
We wish to apply Proposition \ref{r54}.(A) and its Corollary \ref{r55}. 
We need to verify that $\phi b_*b^* F=0$ and that $\phi f_* b^* F=0$.  Both follow, again, from 
Lemma \ref{jhhje}.  This proves diagram (\ref{g00}). That the arrows are direct sum arrows follows from 
Lemma \ref{irrco}.
\end{proof}

The following result identifies the filtered cone of the restriction morphism in (\ref{g00}).
We state the following proposition using morphisms of long exact sequences in cohomology. The reader should have no difficulty
re-writing the stronger version of  (\ref{g11})  that  uses $(R\Gamma (-,-), P^f)$ and distinguished triangles
in $DF(pt)$.

\begin{tm}\la{lezse}
Let $\ms H(-)$ denote either the singular cohomology, the intersection cohomology, or the topologist's intersection cohomology
groups.
 In what follows, for intersection cohomology, when dealing with $Z_s, Z_t$, replace $P^f(1)$ with $P^f$ and $\bullet -2$
with $\bullet -1$. 
We have as isomorphism of long exact sequences of filtered morphims:
\beq\la{g11}
\xymatrix{
\ldots \ar[r] &
\left(\ms{H}^{\bullet -2} (Z_t), P^{f_t}(1) \right) \ar[r]  &
\left(\ms{H}^\bullet (X_t), P^{f_t}\right) \ar[r]   &   \left(\ms{H}^\bullet (X^o_t) P^{f_t}\right) \ar[r] & \ldots  \\
\ldots \ar[r]  & \left(\ms{H}^{\bullet -2} (Z_s), P^{f_s}(1)\right) \ar[r] \ar[u]^-{sp^*_{P^f}}_-\simeq   &
\left(\ms{H}^\bullet (X_s), P^{f_s}\right)   \ar[r] \ar[u]^-{sp^*_{P^f}}_-\simeq  &  \left(\ms{H}^\bullet (X^o_t), P^{f_s}\right)
 \ar[u]^-{sp^*_{P^f}}_-\simeq  \ar[r] & \ldots.
}
\eeq
For  (the topologist's) intersection cohomology, the morphisms are direct sum morphisms for the decomposition  according to irreducible
components (\ref{naadm}).
\end{tm}
\begin{proof}
The plan is to use (\ref{mttc}) in 
 Theorem \ref{sonoloro}, which stems from (\ref{iola}), and    plug (\ref{lki9}) into the result.

 In order to use (\ref{mttc}), we need to verify that the conditions (I) in Theorem \ref{sonoloro} are met.
 Even if it is not necessary, let us note that conditions (II) and (III) are also met in the case where we deal with intersection cohomology  via the use of $IC_X$; however, 
 they are not necessarily met if we deal with singular cohomology via the use of the non semisimple $\rat_X$. 
 
 Conditions (Ia,Ib) are met by virtue of Lemma \ref{f023}. 
 Conditions (Ic,Id) are met by virtue of Lemma \ref{jhhje}.
 
Plug $F= \rat_X[1]$ in (\ref{iola}), use $i^*\rat_X= \rat_{X_s}$ and $i^*a^! \rat_{X_s}=\rat_{Z_s}[-2]$ (cf. \ref{lki9}),
and similarly for $t^*$,
and get the system of isomorphisms of distinguished triangles:
\beq\la{adj1}
\xymatrix{
\ptd{\bullet +1} f_*   \rat_{Z_t}[-2] \ar[r] &  \ptd{\bullet} f_*   \rat_{X_t} \ar[r] & 
b_*  \ptd{\bullet} f_* \rat_{X^o_t} \ar@{~>}[r]. &
\\
\ptd{\bullet +1} f_*   \rat_{Z_s}[-2] \ar[r]  \ar[u]^-\simeq &  \ptd{\bullet} f_*   \rat_{X_s} \ar[r]  \ar[u]^-\simeq & 
b_*  \ptd{\bullet} f_* \rat_{X^o_s} \ar@{~>}[r]  \ar[u]^-\simeq &.
}
\eeq
Then the reader can use (\ref{tsh})
 and (\ref{conv}), to observe that $P^f(-1)$ in (\ref{mttc}) then becomes $P^f(1)$ as in (\ref{g11}), which is thus proved when  $\ms{H}$ is singular cohomology.
 
 Plug $F= IC_X$ in (\ref{iola}), use $i^*[-1]IC_X= IC_{X_s}$ and $i^*a^! IC_{X}=IC_{Z_s}[-1]$ (cf. \ref{lki9}),
and similarly for $t^*$,
and get the system of isomorphisms of distinguished triangles:
\beq\la{adj1bix}
\xymatrix{
\ptd{\bullet +1} f_*   IC_{Z_t}[-1] \ar[r] &  \ptd{\bullet} f_*   IC_{X_t} \ar[r] & 
b_*  \ptd{\bullet} f_* IC_{X^o_t} \ar@{~>}[r]. &
\\
\ptd{\bullet +1} f_*   IC_{Z_s}[-1] \ar[r]  \ar[u]^-\simeq &  \ptd{\bullet} f_*   IC_{X_s} \ar[r]  \ar[u]^-\simeq & 
b_*  \ptd{\bullet} f_* IC_{X^o_s} \ar@{~>}[r]  \ar[u]^-\simeq &.
}
\eeq
Then the reader can use (\ref{tsh})
 and (\ref{conv}) to observe that $P^f(-1)$ in (\ref{mttc})  becomes $P^f$ as in (\ref{g11}) -and that $\bullet  -2$ should be replaced by $\bullet -1$-,  which is thus proved when  $\ms{H}$ is intersection  cohomology. That the  morphisms are direct sum morphisms for the decompositions into irreducible components follows from (\ref{naadm}).
 
 Similarely, if we plug $\m{IC}_X=\oplus_j IC_{X_j}[-\dim X_j]$, we end up with the desired conclusion.
\end{proof}

\begin{rmk}\la{allokvar}$\;$

\ben
\item
The results of this section hold if we replace the Dolbeault moduli space of Higgs bundles for $G=GL_n, SL_n, PGL_n$ for families of curves
with their twisted counterparts of
\ci[Remark 2.1.2]{decomp2018}, where the degree and rank of the Higgs bundles  are coprime.

\item
If $G=GL_n, SL_n$, these twisted Dolbeault moduli spaces are nonsingular, their compactifications
in \S\ref{codomosp} 
are orbifolds, and they can be resolved with boundary a simple normal crossing divisor over the base of the family $S$; see \ci[Thm. 3.1.1.(6)]{decomp2018}.
What follows is certainly redundant in our set-up of the Hitchin morphism, but  may be useful in other set-ups:
the proof of Theorem \ref{lezse} uses conditions (I) in Theorem \ref{sonoloro}; since we are in a simple normal crossing divisor situation, in view of Corollary \ref{vbre}, we could use the variant of conditions (II,III) in Theorem \ref{sonoloro} where one only assumes
that $f_* i^*F[-1]$ satisfies the conclusion of the Hard Lefschetz Theorem. We leave the details to the reader.
\een
\end{rmk}


\end{document}